\documentclass[11pt]{amsart}

\usepackage{a4wide}
\usepackage{amssymb}
\usepackage{latexsym}
\usepackage{amsmath}
\usepackage{amsthm}
\usepackage{amsfonts}
\usepackage{color}
\usepackage{pdfsync}
\usepackage{epsfig}

\newcommand{\I}{\mathcal{I}}
\newcommand{\JJ}{\mathcal{J}}
\newcommand{\R}{\mathbb{R}}
\newcommand{\N}{\mathbb{N}}

\newcommand{\PP}{\mathcal{P}}

\theoremstyle{plain}
\newtheorem{defi}{Definition}[section]
\newtheorem{prop}[defi]{Proposition}
\newtheorem{teo}[defi]{Theorem}
\newtheorem{cor}[defi]{Corollary}
\newtheorem{lema}[defi]{Lemma}

\newtheorem{remark}[defi]{Remark}

\theoremstyle{definition}

\theoremstyle{remark}

\numberwithin{equation}{section}

\begin{document}

\title[]{Fractional truncated Laplacians:\\ representation formula, fundamental solutions\\ and applications}

\author[]{Isabeau Birindelli}

\author[]{Giulio Galise}
\address{
Isabeau Birindelli and Giulio Galise:
Dipartimento de Matematica Guido Castelnuovo, Sapienza Università di Roma,
Piazzale Aldo Moro 5, Roma, ITALIA.
\newline \textit{Email address:} {\tt galise@mat.uniroma1.it, isabeau@mat.uniroma1.it}
}
\author[]{Erwin Topp}
\address{
Erwin Topp:
Departamento de Matem\'atica y C.C., Universidad de Santiago de Chile,
Casilla 307, Santiago, CHILE.
\newline \textit{Email address:} {\tt erwin.topp@usach.cl}
}


\date{\today}

\begin{abstract} 
We introduce some nonlinear extremal nonlocal operators that approximate the, so 
called, truncated Laplacians. For these operators we construct representation formulas that lead to the 
construction of what, with an abuse of notation, could be called \lq\lq fundamental solutions\rq\rq. This, in turn, leads to Liouville type results.\\
The interest is double: on one hand we wish to \lq\lq understand\rq\rq\ what is the right way to define the nonlocal version of the truncated Laplacians, on the other, we introduce nonlocal operators whose nonlocality is  on one dimensional lines, and this dramatically changes the prospective, as is quite clear from the results obtained that often differs significantly with the local case or with the case where the nonlocality is diffused. Surprisingly this is true also for operators that approximate the Laplacian.
\end{abstract}

\maketitle
\section{Introduction}

In the last decades there has been an increasing interest in the comprehension of second order  degenerate elliptic equations.  
The general idea being that new phenomena may occur when the uniform ellipticity condition is replaced by weaker form of ellipticity, while 
other fundamental properties like e.g. the comparison principle may still hold. It would be impossible and far too long to enumerate all the 
works and the \lq\lq kind\rq\rq\ of degeneracies that have been considered: degeneracy may depend of the point of application of the operator, on the 
value of the gradient of the solution, or it may be the case that the operator is simply \lq\lq monotone\rq\rq\,\ i.e. for any couple of symmetric matrices 
$X$ and $Y$
$$X\leq Y\Rightarrow F(X)\leq F(Y).$$
In the realm of nonlocal equations, these very degenerate operators have only just begun to be considered, but they seem to open very interesting and surprising results as will be evident later on, for example in the strong maximum principle of Proposition \ref{propSMP}  or the Liouville Theorem \ref{linearLiouville}.  In order to start a theory on nonlocal degenerate elliptic fully nonlinear operators, one needs to define general operators that are \lq\lq extremal\rq\rq\ among that class.  So that sub or supersolutions of these extremal operators are sub or supersolutions for any degenerate operator. 
We will now define the two  classes of nonlocal extremal operators we will consider in this paper.
In both cases, the fractional order of the operator is cast by $s \in (0,1)$.

We start with the first model, the description is somehow long for an introduction, so we ask for some patience from the reader: let $N \in \N$, $k \in \{ 1,2,...,N \}$, given $\xi \in \mathbb S^{N-1}$, $x \in \R^N$ and $u: \R^N \to \R$, we denote by
\begin{equation}\label{linear}
\begin{split}
\I_\xi u(x) = C_s \mathrm{P.V.} \int_{-\infty}^{+\infty} [u(x + \tau \xi) - u(x)] |\tau|^{-(1 + 2s)}d\tau,
\end{split}
\end{equation}
where $\mathrm{P.V.}$ stands for the Cauchy Principal Value and $C_s = C_{1,s} > 0$ is a normalizing constant related to the well-known fractional Laplacian $-(-\Delta)^s$, see~\eqref{frac} below. Roughly speaking, $\I_\xi$ acts as the one dimensional fractional $2s$-derivative in the direction of $\xi$. 
More precisely, for each $\xi \in \mathbb S^{N-1}$ and with an appropriate modification in the choice of the normalizing constant $C_{s}$ in the definition, $\I_{\xi}$ is identified with the pseudo-differential operator in the Schwartz space defined through the symbol
$$
v \mapsto -|\langle \xi, v \rangle|^{2s}.
$$

An important fact related to the choice of the normalizing constant $C_s$  and to the understanding of the definition of $\I_\xi u(x)$ is the asymptotic 
$$
\I_{\xi}u(x) \to \langle D^2u(x) \xi, \xi \rangle, \quad \mbox{as} \ s \to 1^-,
$$
under suitable regularity assumptions on $u$.
We can now define the extremal operators
\begin{equation}\label{operator}
\I_k^+ u(x) = \max \Big{\{} \sum_{i=1}^{k} \I_{\xi_i} u(x) \ : \ \{ \xi_i \}_{i=1}^{k} \in \mathcal V_k \Big{\}},
\end{equation}
and similarly for $\I_k^-$ taking instead the minimum, where $\mathcal V_k$ is the family of $k$-dimensional orthonormal sets in $\R^N$.
Let us emphasize that these operators are nonlocal, but the nonlocality is in some sense one dimensional. As far as the case $k=1$ is concerned, let us mention that $\I_1^-$ has been recently considered by  Del Pezzo-Quaas-Rossi \cite{DPQR} in order to introduce  the notion of fractional convexity.

\medskip
The second class of operators are instead $k$-dimensionally nonlocal. For $V \in \mathcal V_k$, we denote $ \langle V \rangle$ the $k$-dimensional subspace generated by $V$. Then, for $x \in \R^N$ and $u: \R^N \to \R$ and $V = \{ \xi_1, ..., \xi_k \} \in \mathcal V_k$ we denote
\begin{equation*}
\JJ_{V}u(x) = C_{k,s}\, \mathrm{P.V.} \int_{\R^k} [u(x + \sum_{i=1}^{k} \tau_i \xi_i) - u(x)] (\sum_{i=1}^{k} \tau_i^2)^{-\frac{k + 2s}{2}} d\tau_1 ... d\tau_k.
\end{equation*}
where $C_{k,s} > 0$ is the normalizing constant of the fractional Laplacian in the $k$-Euclidean space (c.f.~\eqref{normconst}).
Using the change of variables formula (see~\cite{EG}), we have the equivalent formulation
\begin{equation*}
\JJ_{V}u(x) = C_{k,s}\, \mathrm{P.V.} \int_{\langle V \rangle} [u(x + z) - u(x)] |z|^{-(k + 2s)}d\mathcal H^k(z),
\end{equation*}
where $\mathcal H^k$ is the $k$-dimensional Hausdorff measure in $\R^N$.

Then, the extremal operator we consider here is
\begin{align}\label{operador2}
\JJ_k^+u(x) = \max_{V \in \mathcal V_k} \JJ_{V}u(x),
\end{align}
and analogously for $\JJ_k^-$ replacing $\max$ by $\min$ in the above definition. Notice that $ \JJ_1^{\pm} = \I_1^{\pm}$ and $\JJ_N^{\pm} = -(-\Delta)^s$. For this reason, concerning $\JJ_k^{\pm}$,  we only concentrate on the cases $1 < k < N$.

\medskip
Clearly for both classes of operators,  in a suitable functional framework, say for bounded smooth functions $u$, $\I_k^\pm u(x) $ and $\JJ_k^\pm u(x)$ converge to  the so called truncated laplacian 
$\PP_k^\pm u(x)$ as $s \to 1$, where 
\begin{equation}\label{truncated1}
\PP_k^+ u(x):= \sum_{i = N - k + 1}^N \lambda_i(D^2u(x))= \max \Big{\{} \sum_{i=1}^{k} \langle D^2 u(x) \xi_i, \xi_i \rangle  \ : \ \{ \xi_i \}_{i=1}^{k} \in \mathcal V_k \Big{\}},
\end{equation}
$\lambda_i(D^2u)\leq\lambda_{i+1}(D^2u)$ being the eigenvalues of $D^2u$ arranged in nondecreasing order, and, mutatis mutandis, similarly for $\PP_k^- u(x)$ which is the sum of the smallest $k$-eigenvalues, we replace $\max$ by $\min$ in the above formula.
The truncated laplacians have received a certain interest, both in geometry and PDE. We wish to remember the works of: Harvey-Lawson \cite{HL1,HL2}, Caffarelli-Li-Nirenberg \cite{CLN}, Capuzzo Dolcetta-Leoni-Vitolo \cite{CDLV}, Blanc-Rossi \cite{BR}  and of two of the authors of this note with Ishii and Leoni \cite{BGI,BGI2,BGL}.
One of the scopes of this paper is to shed some light on different ways of defining generalizations of these extremal degenerate elliptic operators.

The above definitions seem to be natural extensions of the nonlinear second-order operator to the nonlocal setting, in view of the definition of the fractional Laplacian $(-\Delta)^s$. Evaluated on a measurable function $u$ satisfying regularity and growth condition at infinity, its precise definition reads as
	\begin{equation}\label{frac}
	(-\Delta)^s u(x) = - C_{N,s} \mathrm{P.V.} \int_{\R^N} [u(x+z) - u(x)] |z|^{-(N + 2s)}dz,
	\end{equation}
	where P.V. stands for the Cauchy principal value and $C_{N,s} > 0$ is a normalizing constant making $-(-\Delta)^s \to \Delta$ as $s \to 1^-$. See~\eqref{normconst} in the Appendix for details on this constant.

We use the integral formulation of the problem, and then we restrict the evaluation over a class of summable functions along lines with arbitrary direction. A sufficient condition for this is cast by the space
\begin{equation*}
L^1_{1, \sigma} = \Big{\{} u \in L^1_{loc}(\R^N) : \int_{V} \frac{|u(y)| d  \mathcal H^1(y)}{1 + |y|^{1 + \sigma}} <+ \infty \;\, \forall\, V \ \mbox{affine subspace of} \ \R^N,\,\dim(V)=1  \Big{\}}
\end{equation*}
where $\mathcal H^1$ stands for the one dimensional Hausdorff measure in $\R^N$.


Notice that a function $u \in L^1(\R^N)$ may not belong to $L^1_{1,\sigma}$.

\medskip
The first necessity has been to find representation formulas,  at least say for radial functions with completely monotonic profile, for example, for power type functions. 
If we focus on the evaluation of $\I_k^+$ at a function $u(x) = |x|^{-\gamma}$ for $\gamma > 0$, the heuristic makes it reasonable to think that the operator preferably picks a frame $\{ \xi_i \}_{i}$ which includes the direction $\hat x = x/|x|$ (we assume $x \neq 0$), since along this radial direction the one dimensional profile of $u$ shows a sharper convexity. Then, the integral associated to the component $\I_{\hat x} u(x)$ at~\eqref{operator} involves the singularity of $u$ at the origin, which immediately restricts the exponent $\gamma < 1$. 
The mentioned representation formulas is depicted in Corollary~\ref{propPk} below. Concerning the maximal operator $\I_k^+$, the idea discussed above about the preference of the radial direction is confirmed. 

Concerning $\I_k^-$, the representation formula shows that in the case $k < N$ the operator picks a frame which is orthogonal to $\hat x$. 
More intriguing is the case of $\I_N^-$. We start noticing that it does not matches $-(-\Delta)^s$, and in fact $\I_N^-\neq\I_N^+$, while the equality occurs in the limit $s\to1^-$ with the asymptotics $\I_N^\pm \to \Delta$ as $s \to 1^-$.
We prove that for radial functions $u$ with convex, decreasing one dimensional profile, the operator chooses a frame in which all its elements form the same angle with respect to $\hat x$, and therefore we have the beautiful geometric symmetry result
\begin{equation}\label{repI-N}
\I_N^- u(x) = N \I_{\xi^*} u(x) \quad \mbox{for} \ x \neq 0,
\end{equation}
where $\xi^*$ is a unit vector such that $\left\langle \xi^*, \hat x\right\rangle = 1/\sqrt{N}$. 

We would like to mention that our representation formulas are obtained under rather strong monotonicity assumptions on the one dimensional profile of $u$. Such conditions allow to provide a representation formula for every $x \neq 0$, and therefore we believe they can be relaxed if we look  for instance, for the evaluation on bounded domains.

The representation formulas will be used in order to prove Liouville type results, i.e. existence or nonexistence of entire solutions  (or supersolutions) bounded from below. 
First we will consider \lq\lq superharmonic\rq\rq\ functions i.e. 
supersolutions of
\begin{align*}
\I_k^+ u = 0 \quad \mbox{in} \ \R^N.
\end{align*}
When $k=1$ and $s\in[\frac12,1)$ there are no nonconstant supersolutions bounded from below while, in the other cases, such supersolutions do exist. 
Interestingly this result is in contrast to the local second order counter part, and  it is really due to the fractional nature of the operator. This is explained by the existence of a “fundamental solution” of logarithmic profile in the later case.
Roughly speaking, since $s < 1$, there is a “gap” between the order of the operator and the dimensionality, and this room allows us to construct power-type fundamental solutions. We refer to the Appendix for a discussion about the asymptotic behaviour of the exponent of this fundamental solution when we approach the local regime (that is, when $s\rightarrow 1^-$), see Lemma 6.2.

 We also consider semilinear Liouville theorems for the equation
\begin{align*}
\I_k^\pm u + u^p = 0 \quad \mbox{in} \ \R^N.
\end{align*}
These semilinear Liouville theorems usually determine a critical value of the exponent $p$ above which there 
exists supersolutions and below which such nontrivial supersolutions don't exist.

In the case $\I_k^+$,  as it can be seen in Theorem~\ref{teoLiouville+} , and in view of the above discussion, the critical exponent $p$ leading to existence/nonexistence of nontrivial supersolutions for this equation is determined by the exponent of the power-type fundamental solution, which, by the nonlocal nature of the problem, is restricted to be less than $1$. As a consequence, we see that the Liouville result does not meet its local counterpart~\eqref{eqlocal} as $s \to 1^-$, in the sense that the critical exponent of the nonlocal equation diverges to infinity (equivalently the exponent of the fundamental solution vanishes, see the Appendix). %
%
 This is a remarkably nonlocal phenomena that is influenced by the tails of the kernel of the operator more than by its singularity.

\medskip

Concerning $\I_k^-$, the representation formula shows that in the case $k < N$ the operator picks a frame which is orthogonal to $\hat x$. This allows us to conclude the existence of nontrivial supersolutions to
\begin{align}\label{eqnolocal}
\I_k^- u + u^p = 0 \quad \mbox{in} \ \R^N,
\end{align}
for every $p>0$. This phenomena is closely related with its local counterpart presented in~\cite{BGL}.

\medskip

In the case of the equation
\begin{align}\label{eqnolocal2}
\I_N^- u + u^p = 0 \quad \mbox{in} \ \R^N,
\end{align}
let us emphasize that the representation formula~\eqref{repI-N} shows that for $x \neq 0$ the evaluation of the integral operator $\I_N^-$ does not observe possible singularities of $u$ at the origin. Thus, we are able to construct adequate fundamental solutions for $\I_N^-$ (at the expense of a technical redefinition of a power-type function) leading to a Liouville result for equation~\eqref{eqnolocal2} which is more in the direction of classical results, and more interesting, with a critical exponent that passes to the limit as $s \to 1^-$. 
%
%

The local counterpart of these Liouville theorems concerns the  equations
\begin{equation}\label{eqlocal}
\PP_k^{\pm} u + u^p = 0 \quad \mbox{in} \ \R^N.
\end{equation}

This problem was studied by two of the authors and F. Leoni in~\cite{BGL}. The construction of fundamental solutions for $\PP_k^{+}$ follows a careful analysis of the eigenvalues of the Hessian of radial functions and the use of the formula~\eqref{truncated1}. Once fundamental solutions are at disposal, Liouville-type results associated to the so-called Serrin exponent in space dimension $k$, i.e. $\frac{k}{k-2}$, follow the directions of~\cite{CL}. Results concerning $\PP_k^-$ are also provided there.

Concerning the other possible extremal operator $\JJ_k^{\pm}$ ,
we also obtain representation formulas for its evaluation on radial, monotone functions, leading to power-type fundamental solutions for these operators. Here we would like to mention that $\JJ_1^{\pm} = \I_1^{\pm}$, meanwhile $\JJ_N^{\pm} = -(-\Delta)^s$, from which we restrict ourselves to the case in which $k$ is neither $1$ nor $N$.

In view of the definition~\eqref{operador2}, the higher dimensionality of the integrand allows to prove, in the case of $\JJ_k^+$, that the fundamental solutions meet the ones of the $k$-th dimensional fractional Laplacian $-(-\Delta_{\R^k})^s$. This makes the analysis simpler and closer to the local context in the sense that the critical exponent associated to the problem
\begin{equation}
\JJ_k^+ u + u^p = 0 \quad \mbox{in} \ \R^N,
\end{equation} 
meets the critical exponent of~\eqref{eqlocal} as $s \to 1$. 
In particular, this shows that operators $\JJ_k^+$ and $\I_k^+$ are not equivalent, raising an interesting question related to which of them is 
more adequate for applications.

\medskip

The paper is organized as follows: in Section~\ref{sec1} we introduce the notion of viscosity solution and discuss comparison/maximum principles. In Sections~\ref{secrep} and~\ref{secLiouville} we concentrate on $\I_k^{\pm}$: in Section~\ref{secrep} we provide the representation formulas for radial, monotone functions, and in Section~\ref{secLiouville} we present the Liouville-type results for semilinear problems. In Section~\ref{secJ} we discuss the results for $\JJ_k^{\pm}$. Finally, in the Appendix we discuss the asymptotics as $s\to1^-$.

\section{Preliminaries and maximum principles}
\label{sec1}

We start with the notion of viscosity solution for our problem. Here $\I$ may denote $\I = \I_k^{\pm}, \JJ_k^{\pm}$, and $H \in C(\R^N \times \R \times \R^N)$.

We require certain structural assumptions. We will say that a function $u$ is admissible for $\I$ if $u \in L^1_{1, 2s}$ if $\I = \I_{k}^{\pm}$, and $u \in L_{k, 2s}^1$ when $\I = \JJ_{k}^{\pm}$, where for $k > 1$ and $\sigma \in (0,2)$ we denote the set
\begin{equation*}
L^1_{k, \sigma} = \Big{\{} u \in L^1_{loc}(\R^N) : \int_{V} \frac{|u(y)| d\mathcal H^k(y)}{1 + |y|^{k + \sigma}} <+\infty \;\, \forall\, V \ \mbox{affine subspace of} \ \R^N,\,\dim(V)=k  \Big{\}}
\end{equation*}
where $\mathcal H^k$ is the $k$-dimensional Hausdorff measure in $\R^N$.

For viscosity evaluation, we make precise some notation. Given $\xi \in \mathbb S^{N-1}$ we denote
\begin{align*}
& \I_{\xi, \delta} \phi(x) =  C_{1, s}\, \mathrm{P.V.} \int_{-\delta}^\delta [\phi(x + \tau \xi) - \phi(x)] |\tau|^{-(1 + 2s)} d\tau \\
& \I_{\xi}^\delta \phi(x) = C_{1,s} \int_{(-\delta, \delta)^c} [\phi(x + \tau \xi) - \phi(x)] |\tau|^{-(1 + 2s)} d\tau
\end{align*}
and, for each $k = 1,\ldots,N$, $\xi = \{ \xi_i \}_{i=1}^{k} \in \mathcal V_k$, we denote
\begin{align*}
& \JJ_{\xi, \delta} \phi(x) =  C_{k,s}\, \mathrm{P.V.} \int_{B_\delta} [\phi(x + \sum_{i=1}^{k} z_i \xi_i) - \phi(x)] |z|^{-(k + 2s)} dz \\
& \JJ_{\xi}^\delta \phi(x) = C_{k, s} \int_{B_\delta^c} [\phi(x + \sum_{i=1}^{k} z_i \xi_i) - \phi(x)] |z|^{-(k + 2s)} dz,
\end{align*}
where $z=(z_1,\ldots,z_k)\in\mathbb R^k$.

Notice that in the case $k=1$ then $\I_{\xi, \delta} = \JJ_{\xi, \delta}$ and $\I_{\xi}^\delta = \JJ_{\xi}^\delta$ for each $\xi \in \mathbb S^{N-1}$.

The following definition is provided in~\cite{BI}. 
\begin{defi}
	An upper semicontinuous function $u: \R^N \to \R$, admissible with respect to $\I$, is a viscosity subsolution to
	\begin{equation}\label{eqgral}
	\I u + H(x,u,Du) = 0
	\end{equation}	
	at a point $x_0 \in \R^N$ if for every function $\varphi \in C^2(B_\delta(x_0))$, $\delta > 0$,  such that $x_0$ is a local maximum point to $u - \varphi$, then 	\begin{equation*}
	\I(u, \varphi, x_0, \delta) + H(x_0, u(x_0), D\varphi(x_0)) \geq 0,
	\end{equation*}
	where 
	\begin{align*}
	\I_k^{+}(u, \varphi, x_0, \delta)
	& = \max_{ \xi \in \mathcal V_k} \{ \sum_{i=1}^{k} \I_{\xi_i, \delta} \varphi(x_0) + \I_{\xi_i}^\delta u(x_0)\},\\
\JJ_k^{+}(u, \varphi, x_0, \delta)
& = \max_{ \xi \in \mathcal V_k} \{ \JJ_{\xi, \delta} \varphi(x_0) + \JJ_{\xi}^\delta u(x_0)\},
\end{align*} 	
and where the evaluation for $\I = \I_k^-, \JJ_k^-$ is defined replacing the ``max" by ``min" in the respective definitions. 

We define viscosity supersolution and solution in the usual fashion.	
\end{defi}

Notice that $\I_k^- u = -\I_k^+ (-u)$ and analogously for $\JJ_k^{\pm}$. This property is compatible with the viscosity formulation. In fact, classical solutions are viscosity solutions for this definition. As usual, the above notion of solution can be equivalently defined if we ask that $x_0$ is a strict global maximum and/or minimum point, and the well-known stability properties of the viscosity solutions are at disposal here.

As a consequence, comparison principle holds among semicontinuous viscosity sub and supersolutions when the Hamiltonian $H$ satisfies the standard properness/continuity assumptions: if $u, v$ are respectively viscosity subsolution and supersolution to~\eqref{eqgral} in an open set $\Omega \subset \R^N$ (not necessarily bounded or smooth) such that $u \leq v$ in $\Omega^c$, then $u \leq v$ in $\Omega$, see~\cite{BI, BChI}.

%
%
%

The definition above admits unbounded or singular sub and/or supersolutions as soon the nonlocal operator is well-defined (say, belongs to an appropriate class $L^1_\sigma$). . 

\medskip

We now state the basic statement regarding the failure and the validity of the strong maximum/minimum principles for the operators $\I^\pm_k$.

\begin{prop}\label{propSMP} For any $1\leq k<N$ there exist nonconstant smooth solutions of 
\begin{equation}\label{1409eq1}
\I_k^- u \leq 0 \quad \mbox{in} \ \R^N
\end{equation}
which attain their minimum at some point in $\R^N$.  

If $u$ satisfies
		\begin{equation}\label{1409eq2}
		\I_N^- u \leq 0 \quad \mbox{in} \ \R^N, 
		\end{equation}
		in the viscosity sense, and it attains its minimum at some $x_0 \in \R^N$, then $u$ is constant.
\end{prop}		

		\begin{remark}\label{dual}		In a dual fashion, for any $1\leq k<N$ there exist nonconstant smooth solutions of 
\begin{equation}\label{1409eq3}
\I_k^+ u \geq 0 \quad \mbox{in} \ \R^N
\end{equation}
which attain their maximum at some point in $\R^N$.  

If $u$ satisfies
		\begin{equation}\label{1409eq4}
		 \I_N^+ u \geq 0 \quad \mbox{in} \ \R^N, 
		\end{equation}
		in the viscosity sense, and it attains its maximum at some $x_0 \in \R^N$, then $u$ is constant.
%
\end{remark}

\begin{remark}\label{SMP} By the general fact
\begin{equation}\label{1509eq1}
\I^+_ku\leq0\;\;\Rightarrow\;\;\I^-_Nu\leq0\,,
\end{equation} 
we immediately obtain, via Proposition \ref{propSMP}, the validity of the strong minimum  principle  for supersolutions of $\I_k^+u=0$.\\
To see \eqref{1509eq1}, let $\left\{\xi_{N-k+1},\ldots,\xi_N\right\}$ be a $k$-dimensional orthonormal set in $\R^N$ such that
\begin{equation}\label{1509eq2}
\I^+_ku=\sum_{i=1}^k\I_{\xi_{N-k+i}}u\,.
\end{equation}
Without loss of generality we may further assume that $\I_{\xi_{N-k+1}}u\leq \I_{\xi_{N-k+2}}u\leq\ldots\leq \I_{\xi_N}u$. Hence from the inequality $\I^+_ku\leq0$, we infer that $\I_{\xi_{N-k+1}}u\leq0$. Let $\xi_1,\ldots,\xi_{N-k}$ such that $\left\{\xi_1,\ldots,\xi_N\right\}$ is an orthonormal basis of $\R^N$ and such that $\I_{\xi_1}u\leq \ldots\leq \I_{\xi_{N-k}}u$.  To obtain \eqref{1509eq1} it is then sufficient to prove that  $\I_{\xi_{N-k}}u\leq0$. If not, then $\I_{\xi_{N-k}}u>\I_{\xi_{N-k+1}}u$ and
$$
\sum_{i=1}^k\I_{\xi_{N-k+i}}u<\I_{\xi_{N-k}}u+\sum_{i=2}^k\I_{\xi_{N-k+i}}u,
$$  
but this contradicts the maximality of $\left\{\xi_{N-k+1},\ldots,\xi_N\right\}$ in \eqref{1509eq2}.

\medskip

In a similar way and using Remark \ref{dual}, we infer that the strong maximum principle for subsolution of $\I^-_ku=0$ holds.
\end{remark}
\begin{proof}[Proof of Proposition \ref{propSMP}]
Let $\varphi$ be a nonconstant smooth and bounded function of one variable which attains the minimum at some point in $\R$. Consider $\varphi$ as a function of $N$ variables just by setting $u(x):=\varphi(x_N)$. It is clear that $u$ is a nontrivial function attaining its minimum at some point in $\R^N$. If $\left\{e_i\right\}_{i=1}^N$ denote the canonical basis in $\R^N$, then
for any $x\in\R^N$ and any $\tau\in\R$ we have
$$
u(x+\tau e_i)=u(x)\quad\text{for $i=1,\ldots,N-1$}.
$$ 
Hence $\I_{e_i}u(x)=0$ for any $i=1,\ldots,N-1$ and 
\begin{equation*}
\I_k^-u(x)\leq\sum_{i=1}^k\I_{e_i}u(x)=0\quad \mbox{in} \ \R^N.
\end{equation*}
This concludes the first part of the proof.

For the second part, we use the argument of propagation of maxima through the support of the kernel of the nonlocal operator, see~\cite{C}. 

	Let $y \in \R^N$ and denote $d_0 = |y - x_0|$. Since $x_0$ is a minimum point for $u$ we can use the constant function equal to $u(x_0)$ as test function. For $\delta < \frac{d_0}{\sqrt{N}}$ and let $V = \{ \xi_i \}_{i=1}^N \in \mathcal V_N$ attaining the minimum at $\I_N^- (u, u(x_0), x_0, \delta)$, from which we can write
	\begin{equation*}
	\sum_{i=1}^N \I_{\xi_i}^\delta u(x_0) \leq 0.
	\end{equation*}
	
	In particular, since $x_0$ is a global minimum for $u$, for each $i$ we have
	\begin{equation*}
	\I_{\xi_i}^\delta u(x_0) = 0.
	\end{equation*}
	
	Thus, by the semicontinuity of $u$ we have $u(x_0 + \tau \xi_i) = u(x_0)$ for all $|\tau| \geq \delta$. Since $\delta$ is arbitrarily  small, then the previous equality holds for any $\tau\in\R$.
	
	Now, since $V$ is a basis of $\R^N$, there exists at least one $\xi_i \in V$ such that 
	$$
	\left|\langle \widehat{y - x_0}, \xi_i \rangle\right| \geq 1/\sqrt{N},
	$$
	where $\widehat{y - x_0}=\frac{y - x_0}{|y - x_0|}$. From this there exists $\tau$ such that $x_1 := x_0 + \tau \xi_i$ simultaneously satisfies $u(x_1) = u(x_0)$ and $d_1 := |x_1 - y| \leq d_0 \sqrt{1 - N^{-1}}$.
	
	Using the same argument above but with $x_1$ and $d_1$ replacing $x_0$ and $d_0$, it is possible to find $x_2 \in \R^N$ and $d_2 > 0$ such that $u(x_2) = u(x_0)$ and $|x_2 - y| = d_2 \leq d_1 \sqrt{1 - N^{-1}}$. Then, repeating this argument, we find a sequence $(x_k)$ such that $u(x_k) = u(x_0)$ and $x_k \to y$. By lower semicontinuity, we conclude that $u(y)\leq u(x_0)$ and then $u(y)=u(x_0)$, $x_0$ being the global minimum point of $u$. Since $y$ is arbitrary we get the result. 
\end{proof}


\section{Representation formula for { $\I_k^{\pm}$}.}\label{secrep}

We have already defined the linear operators $\I_\xi$ in~\eqref{linear}. At this point 
it is important to mention that~\eqref{linear} can be equivalently written as
\begin{equation}\label{linear2}
\begin{split}
\I_{\xi}u(x) = & \frac{C_s}{2} \int_{-\infty}^{+\infty} [u(x + \tau \xi) + u(x - \tau \xi) - 2u(x)] |\tau|^{-(1 + 2s)}d\tau \\
= & C_s \int_{0}^{+\infty} [u(x + \tau \xi) + u(x - \tau \xi) - 2u(x)] \tau^{-(1 + 2s)}d\tau.
\end{split}
\end{equation}

 Here and in what follows we use the following notation: for $x \neq 0$, denote $\hat x = x/|x|$ and denote $V_x = \left\langle  \{ \hat x \}\right\rangle^{\perp}$ the orthogonal subspace to $\hat x$. Given a subspace $V$, we denote $\pi_V$ the projection onto $V$. Then
\begin{lema}\label{lema0}
	Let $\xi \in \mathbb S^{N - 1}$, $x \in \R^N$ and $u \in C^2(\mathbb R^N) \cap L_{1,2s}^1$. Then
	\begin{itemize}
		\item[(a)] $\I_{\xi}u(x) = \I_{-\xi}u(x)$.
		
		\medskip
		
		\item[(b)] If $R$ is any rotation matrix in $\R^N$ and if we denote $\tilde u(x) = u(Rx)$, then
		\begin{equation*}
		\I_{\xi}\tilde u(x) = \I_{R\xi} u(Rx).
		\end{equation*}
		
		\medskip
		
		\item[(c)] If $u$ is radial, that is $u(x) = g(|x|)$ for some real valued function $g$, then
		$$
		\I_{\xi}u(x) = \I_{\xi_x}u(x),
		$$
		where $\xi_x = \pi_{V_x}(\xi) - \langle \xi, \hat x \rangle \hat x$ is the unit vector, symmetric to $\xi$ with respect to the hyperplane $V_x$. 
		
		\medskip
		
		\item[(d)] If $u$ is radial and $R: \R^N \to \R^N$ is a rotation matrix leaving invariant $V_x$, then
		\begin{equation*}
		\I_{\xi}u(x) = \I_{R \xi}u(x).
		\end{equation*}
	\end{itemize}
	\end{lema}
	
	\begin{proof}
		The proof of $(a)$ and $(b)$ are immediate, and do not require $u$ to be radial. 
		For $(c)$, we see that
		\begin{align*}
		|x + \tau \xi_x|^2 & = ||x| - \tau \langle \xi, \hat x \rangle|^2 + |\tau \pi_{V_x}(\xi)|^2,
		\end{align*}
		and using the symmetry of the kernel, we make the change of variables $\tau = -\tau$, and noticing that $\xi = \pi_{V_x}(\xi) + \langle \xi, \hat x \rangle\hat x$ we conclude the result.
		
		For $(d)$, we notice that
		$$
		|x + \tau R \xi|^2 = |x + \tau \langle \xi, \hat x \rangle \hat x|^2 + |R \pi_{V_x}(\xi) \tau|^2, 
		$$
		and using that a rotation matrix is an isometry, we conclude the result.
	\end{proof}

\begin{remark}\label{rmklema0}
	By the previous lemma, for every radial function $u$ and every orthonormal frame $\{ \xi_i \}_{i=1}^{k}$ the definition of the operator $\I_k^{\pm}u(x)$ can be taken in such a way that the angle between $x$ and each $\xi_i$ is in $[0, \pi/2]$.
\end{remark}

Now we present the main technical result of this section.

\begin{lema}\label{lemaI}
	Assume $u(x) = \tilde g(|x|^2)$ is such that $u \in L^1_{1,2s}$. For $x \neq 0$ and $\theta \in [0,1]$ we denote
	$$
	I(|x|, \theta) := C_s |x|^{-2s} \int_{0}^{+\infty} \frac{\tilde g(|x|^2(1 + \tau^2 + 2 \tau \theta)) + \tilde g(|x|^2(1 + \tau^2 - 2 \tau \theta)) - 2 \tilde g(|x|^2)}{\tau^{1 + 2s}}d\tau.
	$$

	\begin{itemize}
		\item If $\tilde g$ is convex and $k = 1,..., N-1$, then
		\begin{align}\label{mainineqk-}
		\I_k^- u(x) = k I(|x|, 0).
		\end{align}
		
		\item If $\tilde g''$ is convex, then
		\begin{align}\label{mainineqN-}
		\I_N^- u(x) = N I(|x|, \frac{1}{\sqrt{N}}).
		\end{align}
		
		\item If $\tilde g, \tilde g''$ are convex, for all $k = 1,..., N$ we have
		\begin{align}\label{mainineq}
		\I_k^+ u(x) = I(|x|, 1) + (k - 1) I(|x|, 0).
		\end{align}
	\end{itemize}
\end{lema}

\begin{proof}
	For $a \geq b \geq 0$, let $h:[-1,1] \to \R$ be the function 
	$$
	h(t) = \tilde g(a + bt) + \tilde g(a - bt),
	$$
	and $p: [0,1] \to \R$ defined as $p(t) = h(\sqrt{t})$.
	
	Note that $h, h''$ are even. If $\tilde g$ is convex, so is $h$. From this, $0$ is a minimum point for $h$ and $h$ is nondecreasing in $[0,1]$. The same analysis in the case $\tilde g''$ is convex implies $h''$ is convex and nondecreasing in $[0,1]$.
	
		 	\medskip

	We start with~\eqref{mainineqk-}. Using the monotonicity of $h$, in particular $h(0)\leq h(1)$, for each $a \geq b \geq 0$ we get
	\begin{align}\label{hola}
	2 \tilde g(a) \leq \tilde g(a + b) + \tilde g(a - b).
	\end{align}
	
	Take an orthonormal set $\{ \xi_i \}_{i=1}^k$ and $\tau > 0$. Using the last inequality with $a = |x|^2 + \tau^2$ and $b = 2 \tau |\langle x, \xi_i \rangle|$ we have
	$$
	2 \tilde g(|x|^2 + \tau^2) \leq \tilde g(|x|^2 + \tau^2 + 2 \tau |\langle x, \xi_i \rangle|) + \tilde g(|x|^2 + \tau^2 - 2 \tau |\langle x, \xi_i \rangle|).
	$$
	
	Substracting $2\tilde g(|x|^2)$ in both sides, multiplying by the factor $\tau^{-(1 + 2s)}$, integrating from $0$ to $+\infty$, and summing-up in $i=1,\ldots,k$, we conclude that
	$$
	kI(|x|, 0) \leq \sum_{i=1}^{k} \I_{\xi_i} u(x).
	$$
	
	Since $k < N$, we can select an orthonormal set such that $\langle x, \xi_i \rangle = 0$ for all $i$ and the lower bound is attained, from which we arrive at~\eqref{mainineqk-}.		
	
	\medskip

	Now we continue with~\eqref{mainineqN-}. For this, we use that $\tilde g''$ is convex. Observe that $h'(0) = 0$. Now, $p'(t) = \frac{1}{2\sqrt{t}}h'(\sqrt{t})$ and then
	\begin{align*}
	p''(t) & = \frac{1}{4t} h''(\sqrt{t}) - \frac{1}{4t^{3/2}} (h'(\sqrt{t}) - h'(0)) \\
	& = \frac{1}{4t} h''(\sqrt{t}) - \frac{1}{4t^{3/2}} \int_{0}^{\sqrt{t}}h''(\theta) d\theta  \\
	& = \frac{1}{4t} \int_{0}^{1} [h''(\sqrt{t}) - h''(\sqrt{t}\theta)] d\theta\,.
	\end{align*}
	Since $h''$ is nondecreasing in $[0,1]$ we obtain $p''\geq 0$, which shows that $p$ is convex in $[0,1]$.
	
	Consider the simplex $\Lambda = \{ \lambda = (\lambda_1,..., \lambda_k) : \lambda_i \geq 0 \ \mbox{for} \ i=1,\ldots,k\,, \ \sum_{i=1}^{k} \lambda_i = 1\}$ and let $P: \Lambda \to \R$ given by
	$$
	P(\lambda) = \sum_{i = 1}^k p(\lambda_i).
	$$

Let $\lambda = (\lambda_1, ..., \lambda_k) \in \Lambda$. Using the convexity of $p$,  we can write
$$
p(\lambda_i) \geq p(\frac{1}{k}) + p'(\frac{1}{k})(\lambda_i - \frac{1}{k})
$$	 	
for each $i$. Then, we conclude that
$$
P(\lambda) \geq k p(\frac{1}{k}). 
$$

In particular, since $(\frac{1}{k}, ..., \frac{1}{k}) \in \Lambda$, we get
$$
\min \limits_{\lambda \in \Lambda} P(\lambda) = k p(\frac{1}{k}) = k \Big{(} \tilde g(a + b \frac{1}{\sqrt{k}}) + \tilde g(a - b \frac{1}{\sqrt{k}}) \Big{)}.
$$

When $k=N$, for each orthonormal set $\{ \xi_i \}_{i=1}^{N}$ we have $\sum_{i=1}^{N} \langle \hat x, \xi_i \rangle^2 = 1$, where $\hat x=\frac{x}{|x|}$. Using the last equality with $a = |x|^2 + \tau^2$, $b = 2|x|\tau $ and $\lambda_i = |\langle \hat x, \xi_i \rangle|^2$ we conclude that
\begin{align*}
& \sum_{i=1}^{N} \tilde g(|x|^2 + \tau^2 + 2\tau |x| |\langle \hat x, \xi_i \rangle|) + \tilde g(|x|^2 + \tau^2 - 2\tau |x| |\langle \hat x, \xi_i \rangle|) \\
\geq & N \Big{(}\tilde g(|x|^2 + \tau^2 + 2\tau |x| \frac{1}{\sqrt{N}}) + \tilde g(|x|^2 + \tau^2 - 2\tau |x| \frac{1}{\sqrt{N}}) \Big{)}.
\end{align*}

Again, substracting $2N\tilde g(|x|^2)$ in both sides, multiplying by $\tau^{-(1 + 2s)}$ and integrating, we see that
\begin{equation*}
NI(|x|, \frac{1}{\sqrt{N}}) \leq \sum_{i=1}^{N} \I_{\xi_i} u(x).
\end{equation*}

The infimum is attained. For this, let  consider $O: \R^N \to \R^N$ the orthonormal map so that $ O \hat x = \frac{1}{\sqrt{N}} \sum_{i=1}^{N} e_i$, where $\{ e_i \}_{i=1}^{N}$ the standard basis in $\R^N$. Set $\xi_i=O^{-1}e_i$. 

Then, using the rotation invariance of the operator, together with the radiality of the function we conclude that
\begin{equation*}
\sum_{i=1}^{N} \I_{\xi_i} u(x) = \sum_{i=1}^{N} \I_{O^{-1} e_i} u(x) = \sum_{i=1}^{N} \I_{e_i} u(O x) = N I(|x|, \frac{1}{\sqrt{N}}).
\end{equation*}

Now we deal with~\eqref{mainineq}. Let $\{ e_i \}_{i=1}^k $ the standard basis of $\R^k$. Since $P$ is convex in $\Lambda$, then we have
	$$
	\max_{\lambda \in \Lambda} P(\lambda) = \max_{i} P(e_i).
	$$	
	
	Observe that $P(e_i) = p(1) + (k - 1)p(0)$ for each $i$, from which we conclude that
	\begin{equation}\label{maxp}
	\max_{\lambda \in \Lambda} \sum_{i=1}^k [\tilde g(a + b\sqrt{\lambda_i}) + \tilde g(a - b\sqrt{\lambda_i})] = \tilde g(a + b) + \tilde g(a - b) + 2(k-1)\tilde g(a).
	\end{equation}

	Let $x \neq 0$ and $\{ \xi_i \}_{i=1}^k$ and orthonormal set in $\R^N$. 
	Let $\rho = \sum_{i=1}^{k} \langle \hat x, \xi_i \rangle^2 \leq 1$ and assume that $\rho > 0$. Denote $\lambda_i = \frac{\langle \hat x, \xi_i \rangle^2}{\rho}$.
	Then, for each $\tau > 0$, by~\eqref{maxp} with $a = |x|^2 + \tau^2$ and $b = 2 \sqrt{\rho} |x| \tau$, and using the monotonicity of $h$, we have
	\begin{align*}
	& \sum_{i=1}^k [\tilde g(|x|^2 + \tau^2 + 2 |x|\tau |\langle \hat x, \xi_i \rangle|) + \tilde g(|x|^2 + \tau^2 - 2 |x|\tau |\langle \hat x, \xi_i \rangle|)] \\
	= & \sum_{i=1}^k [\tilde g(|x|^2 + \tau^2 + 2 |x|\tau \sqrt{\rho \lambda_i}) + \tilde g(|x|^2 + \tau^2 - 2 |x|\tau \sqrt{\rho \lambda_i})] \\
	\leq & \tilde g(|x|^2 + \tau^2 + 2 \sqrt{\rho}|x|\tau) + \tilde g(|x|^2 + \tau^2 - 2 \sqrt{\rho}|x|\tau) + 2(k - 1) \tilde g(|x|^2 + \tau^2) \\
	\leq & \tilde g(|x|^2 + \tau^2 + 2 |x|\tau) + \tilde g(|x|^2 + \tau^2 - 2 |x|\tau) + 2(k - 1) \tilde g(|x|^2 + \tau^2)\,. \\
	\end{align*}
	
	Thus, we arrive at
	\begin{align*}
	& \sum_{i=1}^k [\tilde g(|x|^2 + \tau^2 + 2 |x|\tau |\langle \hat x, \xi_i \rangle|) + \tilde g(|x|^2 + \tau^2 - 2 |x|\tau |\langle \hat x, \xi_i \rangle|) - 2\tilde g(|x|^2)] \\
	\leq & \Big{(}\tilde g(|x|^2 + \tau^2 + 2 |x|\tau) + \tilde g(|x|^2 + \tau^2 - 2 |x|\tau) - 2 \tilde g(|x|^2) \Big{)} + 2(k - 1) \Big{(}\tilde g(|x|^2 + \tau^2) - \tilde g(|x|^2) \Big{)},
	\end{align*}
	from which, after multiplying by the kernel $\tau^{-(1 + 2s)}$ and integration, we get
	\begin{align}\label{last1}
	\sum_{i=1}^{k} \I_{\xi_i} u(x) \leq I(|x|, 1) + (k  -1)I(|x|, 0).
	\end{align}

	When $\rho = 0$, we use again the monotonicity of $h$ (with $a = |x|^2 + \tau^2$ and $b = 2\tau |x|$) to conclude 
	 	\begin{align*}
	 	& \sum_{i=1}^k [\tilde g(|x|^2 + \tau^2 + 2 |x|\tau |\langle \hat x, \xi_i \rangle|) + \tilde g(|x|^2 + \tau^2 - 2 |x|\tau |\langle \hat x, \xi_i \rangle|) - 2\tilde g(|x|^2) - 2\tilde g(|x|^2)] \\
	 	= & 2k [\tilde g(|x|^2 + \tau^2) - \tilde g(|x|^2)] \\
	 	\leq & \Big{(}\tilde g(|x|^2 + \tau^2 + 2 |x|\tau) + \tilde g(|x|^2 + \tau^2 - 2 |x|\tau) - 2 \tilde g(|x|^2) \Big{)} + 2(k - 1) \Big{(}\tilde g(|x|^2 + \tau^2) - \tilde g(|x|^2) \Big{)},
	 	\end{align*}
	 	which leads us to~\eqref{last1} as well. Noting that if we pick an orthonormal set $\{ \bar \xi_i \}_{i=1}^k$ such that $\hat x = \bar \xi_1$, we see that
	 	$$
	 	\sum_{i=1}^{k} \I_{\bar \xi_i} u(x) = I(|x|,1) + (k - 1)I(|x|,0)
	 	$$
	 	and this concludes~\eqref{mainineq}.
	 	\end{proof}

 As a consequence of the above result we have the following representation formulas.
\begin{teo}\label{propPk}
		Assume $u(x) = \tilde g(|x|^2)$ with $\tilde g$ satisfying the assumptions of Lemma~\ref{lemaI}. Let $x \neq 0$ and denote $\hat x = x/|x|$.
	
	\begin{itemize}
		\item[$(i)$] For all $N, k \in \N$ with $1 \leq k \leq N$ we have 
		\begin{equation*}
		\I_k^+u(x) = \I_{\hat x}u(x) + (k - 1) \I_{x^\perp}u(x),
		\end{equation*}
		where $x^\perp \in V_x$ with $|x^\perp|=1$.
		
		\medskip
		
		\item[$(ii)$] If $1 \leq k < N$ we have
		\begin{equation*}
		\I_k^-u(x) = k \I_{x^\perp}u(x),
		\end{equation*}
		where $x^\perp$ is as in the previous point, and
		\begin{equation*}
		\I_N^-u(x) = N \I_{\xi^*}u(x),
		\end{equation*}
		where $\xi^* \in \R^N$ is a unit vector such that $\langle \hat x, \xi^* \rangle = \frac{1}{\sqrt{N}}$.
	\end{itemize}
\end{teo}

\begin{remark}\label{rmkg}
	It is easy to see that examples of functions $\tilde g$ satisfying the above assumptions are $\tilde g(t) = {\sqrt{t}}^{-\gamma}$ with $\gamma \in (0,1)$,  $\tilde g(t) = {\left(a + \sqrt{t}\right)}^{-\gamma}$, $\tilde g(t)={\left(a+t\right)}^{-\gamma}$, $\tilde g(t)=e^{-at}$ for $a > 0$ and $\gamma > 0$.  Another example is the function $\tilde g(t) = -{\sqrt{t}}^\gamma$ for $\gamma \in (0, 2s)$.
\end{remark}

\subsection{Computation on power-type functions.}

We start with the following lemma that can be found in~\cite{BCGJ}, but that we present here for the readers convenience.
\begin{lema}\label{lema2}
	For $\gamma>0$, denote $v_\gamma(x) = |x|^\gamma$. Then for $s \in (0, 1)$ and $\gamma\in(0, 2s)$, there exists a constant $\hat c(\gamma) \in \R$  such that 
	\begin{equation*}
	\I_{\hat x} v_\gamma(x) = \hat c(\gamma)|x|^{\gamma - 2s} \quad \mbox{for all} \ x \neq 0
	\end{equation*}
	and
	\begin{itemize}
		\item $\hat c(\gamma)< 0$ if $\gamma \in \left(0, (2s - 1)_+\right)$
		
		\item $\hat c(\gamma) = 0$ if $\gamma = (2s - 1)_+ $
		
		\item $\hat c(\gamma) > 0$ of $\gamma \in ((2s - 1)_+, 2s)$.
	\end{itemize}
	Similarly, for $\gamma\in(0,1)$, denote $w_\gamma(x) = |x|^{-\gamma}$. Then for $s \in (0, 1)$ there exists a constant $\hat c(\gamma) \in \R$  such that 
	\begin{equation*}
	\I_{\hat x} w_\gamma(x) = \hat c(\gamma)|x|^{\gamma - 2s} \quad \mbox{for all} \ x \neq 0
	\end{equation*}
	and 
	\begin{itemize}
		\item $\hat c(\gamma)< 0$ if $\gamma \in \left(0, (1-2s)_+\right)$
		
		\item $\hat c(\gamma) = 0$ if $\gamma = (1-2s)_+$
		
		\item $\hat c(\gamma) > 0$ of $\gamma \in ((1-2s)_+,1)$.
	\end{itemize}
\end{lema}

\begin{proof}
We only consider the case $\I_{\hat x} v_\gamma(x)$, since the proof concerning $\I_{\hat x} w_\gamma(x)$  follows the same ideas.
	We notice that $\I_{\hat x} v_\gamma(x)$ is well-defined since $v_\gamma\in L^1_{1,2s}$. We have
	\begin{align*}
	\I_{\hat x} v_\gamma(x)  = C_s \mathrm{P.V.} \int_{-\infty}^{+\infty} [|x + \tau \hat x|^\gamma - |x|^\gamma] |\tau|^{-(1 + 2s)}d\tau 
	= C_s |x|^{\gamma - 2s} I, 
	\end{align*}
	where
	$$
	I := \mathrm{P.V.} \int_{-\infty}^{+\infty} [|1 + \tau|^\gamma - 1] |\tau|^{-(1 + 2s)}d\tau.
	$$
	
	We split the last integral as
	$$
	I = \int_{-\infty}^{-1} [|1 + \tau|^\gamma - 1] |\tau|^{-(1 + 2s)}d\tau + \mathrm{P.V.} \int_{-1}^{+\infty} [|1 + \tau|^\gamma - 1] |\tau|^{-(1 + 2s)}d\tau,
	$$
	and using the change of variables $1 + \tau = -e^{z}, z \in \R,$ for the first integral, and $1 + \tau = e^z, z \in \R,$ for the second, we obtain
	\begin{align*}
	I = & \int_{-\infty}^{+\infty} [e^{z\gamma} - 1] (1 + e^z)^{-(1 + 2s)} e^z dz + \mathrm{P.V.} \int_{-\infty}^{+\infty} [e^{\gamma z} - 1] |e^z - 1|^{-(1 + 2s)} e^zdz \\
	= & 2^{-2s} \int_{-\infty}^{+\infty} e^{z((\gamma + 1)/2 -s)} \sinh(\gamma z /2) (\cosh(z/2))^{-(1 + 2s)} dz \\
	& + 2^{-2s} \mathrm{P.V.} \int_{-\infty}^{+\infty} e^{z ((\gamma + 1)/2 -s )}\sinh(\gamma z /2)|\sinh(z/2)|^{-(1 + 2s)}dz\,.
	\end{align*}
	
	Notice that $I = 0$ when $\frac{\gamma + 1}{2} - s= 0$, that is $\gamma = 2s-1$.
	Using the symmetry of the integral $I$ we obtain the result with
	$$\hat c(\gamma)=C_s \mathrm{P.V.} \int_{-\infty}^{+\infty} [|1 + \tau|^\gamma - 1] |\tau|^{-(1 + 2s)}d\tau.$$
\end{proof}

\medskip

Using Theorem~\ref{propPk} we have the following identity
\begin{prop}\label{propbargamma}
Let $\gamma \in (0,1)$ and denote $w_\gamma(x) = |x|^{-\gamma}$ for $x \neq 0$. Then
\begin{equation}\label{0604eq1}
\I_k^+w_\gamma(x) = c_k(\gamma) |x|^{-(\gamma + 2s)}, 
\end{equation}
where $c_k(\gamma) = \hat c(\gamma) + (k - 1)c^\perp(\gamma)$ with
\begin{align*}
\hat c(\gamma) & := C_s \mathrm{P.V.} \int_{-\infty}^{+\infty} [|1 + \tau|^{-\gamma} - 1] |\tau|^{-(1 + 2s)} d\tau \\
c^\perp(\gamma) & := 2 C_s \int_{0}^{+\infty} [(1 + \tau^2)^{-\gamma/2} - 1] \tau^{-(1 + 2s)} d\tau.
\end{align*}

 For $k \geq 1$, the function $c_k:(0,1) \to \R$ satisfies $c_k(0^+) = 0$, $c_k(1^-) = +\infty$, it is strictly convex in $(0,1)$ and there exists a unique $\bar \gamma=\bar\gamma(k,s) \in (0,1)$ such that $c_k(\bar \gamma) = 0$ in the following cases: 
\begin{itemize}
	\item [(i)] $k=1$ and $s\in(0,\frac12)$
	\item[(ii)] $k\geq2$ and $s\in(0,1)$.
\end{itemize}
\end{prop}
\begin{proof}
Formula \eqref{0604eq1} follows directly by the characterization provided in Theorem~\ref{propPk} and the fact that for each $x \neq 0$ we have
\begin{align*}
\I_{\hat x} w_\gamma(x) & = \hat c(\gamma) |x|^{-(\gamma + 2s)} \\
\I_{\hat x^\perp} w_\gamma(x) & = c^\perp(\gamma) |x|^{-(\gamma + 2s)},
\end{align*}
where we have used the homogeneity of the nonlocal operator and the function $w$. 

Using Dominated Convergence Theorem, for each $k$ we have
$$
c_k(\gamma) \to 0 \quad \mbox{as} \ \gamma \to 0^+.
$$

Hence, defining $c_k(0) = 0$, we have $c_k:[0,1) \to \R$ is a continuous function.

On the other hand, noticing that $c^\perp(\gamma)$ is uniformly bounded for $\gamma \in (0,1)$ and that $\hat c(\gamma) \to +\infty$ as $\gamma \to 1^-$, we have
$
c_k(\gamma) \to +\infty 
$
as $\gamma \to 1^-$.

In addition, $c_k \in C^2(0,1)$ and for $\gamma \in (0,1)$ we see that
\begin{align*}
c_k'(\gamma) = & -C_s\left( \int_{0}^{+\infty} \Big{[} |1 + \tau|^{-\gamma} \ln |1 + \tau| + |1 - \tau|^{-\gamma} \ln |1 - \tau| \Big{]} \tau^{-(1 + 2s)} d\tau\right. \\
&\qquad\quad \left.+ (k - 1) \int_{0}^{+\infty} (1 + \tau^2)^{-\gamma/2} \ln (1 + \tau^2) \tau^{-(1 + 2s)} d\tau\right) \\
c_k''(\gamma) = &\;\, C_s\left( \int_{0}^{+\infty} \Big{[} |1 + \tau|^{-\gamma} \ln^2 |1 + \tau| + |1 - \tau|^{-\gamma} \ln^2 |1 - \tau| \Big{]} \tau^{-(1 + 2s)} d\tau\right. \\
& \qquad\quad \left.+ \frac{k - 1}{2} \int_{0}^{+\infty} (1 + \tau^2)^{-\gamma/2} \ln^2 (1 + \tau^2) \tau^{-(1 + 2s)} d\tau\right)
\end{align*}
and from here we clearly have $c_k'' > 0$. Hence $c_k$ is a convex function in $[0,1)$.

Now we prove the existence of a unique $\bar \gamma\in(0,1)$ such that $c_k(\gamma)<0 $ for $\gamma\in(0,\bar\gamma)$, $c_k(\bar\gamma)=0 $ and $c_k(\gamma)>0$ for $\gamma\in(\bar\gamma,1)$.

The case $k=1$ and $s\in(0,\frac12)$ trivially follows from Lemma \ref{lema2}, since 
$$
c_1(1-2s)=\hat c(1-2s)=0.
$$
In this case $\bar\gamma=1-2s$,  $c_1(\gamma)<0 $ if $\gamma\in(0,1-2s)$ and $c_1(\gamma)>0 $ if $\gamma\in(1-2s,1)$.

In what follows we assume $k\geq2$. If $s\in(0,\frac12)$ we have
$$
c_k(1-2s)=\hat c(1-2s)+c^\perp(1-2s)=c^\perp(1-2s)<0.
$$
Then, by   convexity of $c_k$,  there exists a unique $\bar\gamma\in(1-2s,1)$ such that $c_k(\bar\gamma)=0$. Moreover $c_k(\gamma)<0$ for $\gamma\in(0,\bar\gamma)$ and $c_k(\gamma)>0$ for $\gamma\in(\bar\gamma,1)$.

Now we consider $s\in[\frac12,1)$. It is easy to see that $c_k'(0^+)$ exists and we have the expression
\begin{align*}
c_k'(0^+) = & -C_s\left(\int_{0}^{+\infty} \ln |1 - \tau^4| \tau^{-(1 + 2s)} d\tau\right. \\
& \qquad\quad\left.+\, (k - 2) \int_{0}^{\infty} \ln (1 + \tau^2) \tau^{-(1 + 2s)} d\tau\right) \quad \mbox{for} \ k \geq 2.
\end{align*}

We claim that 
\begin{equation}\label{daMatteo}
{ c_k'(0^+) < 0 \quad \mbox{for all} \  k\geq2.}
\end{equation}
From \eqref{daMatteo} we easily obtain the result, again by means of the convexity of $c_k$.

To complete the proof it remains to show \eqref{daMatteo}. Since $c_{k+1}'(0^+)< c_k'(0^+)$ for any $k\geq2$, it is then sufficient to prove the claim for $k=2$. 

Note that 
$$
c_2'(0^+)=-\frac{C_s}{2}F(s),
$$
where
\begin{equation}\label{2603eq1}
F(s)=\int_0^{+\infty}\ln|1 - \tau^2| \tau^{-(1 + s)} d\tau.
\end{equation}

The function $F:[1/2,1] \to \R$ is well defined, and we shall prove that $F(s)>0$. 

A straightforward computation leads us to
\begin{equation}\label{2603eq2}
F(1)=0.
\end{equation}
Moreover for any $s\in\left[\frac12,1\right]$ and for a.e. $\tau\in(0,+\infty)$
\begin{equation}
\begin{split}\label{2603eq3}
\left|\frac{\ln|1 - \tau^2|}{ \tau^{1 + s}}\right|&\leq\left|\ln|1 - \tau^2|\right|\max\left\{\frac{1}{\tau^2},\frac{1}{\tau^{3/2}}\right\}\in L^1\left((0,+\infty)\right) \\
\left|\frac{\partial}{\partial s}\frac{\ln|1 - \tau^2|}{ \tau^{1 + s}}\right|&=\left|\frac{\ln|1 - \tau^2|\ln\tau}{ \tau^{1 + s}}\right|\\
&\leq\left|\ln|1 - \tau^2|\ln\tau\right|\max\left\{\frac{1}{\tau^2},\frac{1}{\tau^{3/2}}\right\}\in L^1\left((0,+\infty)\right)
\end{split}
\end{equation}
By \eqref{2603eq3}, $F\in C^1\left(\left[\frac12,1\right]\right)$ and via integrations by parts we obtain
\begin{equation}\label{2603eq4}
\begin{split}
F'(s)&=-\int_0^{+\infty}\frac{\ln|1 - \tau^2|\ln\tau}{ \tau^{1 + s}}d\tau\\
&=\frac2s\int_0^{+\infty}\frac{\tau^{1-s}\ln\tau}{1-\tau^2}d\tau-\frac1sF(s)\\
&\leq-\frac2sI-\frac1sF(s),
\end{split}
\end{equation}
where 
$$
\int_0^{+\infty}\frac{\tau^{1-s}\ln\tau}{1-\tau^2}d\tau\leq-I:=\int_0^1\frac{\tau\ln\tau}{1-\tau^2}d\tau+\int_1^{+\infty}\frac{\ln\tau}{1-\tau^2}d\tau < 0.
$$

From \eqref{2603eq4} we have
$$
(sF(s))'\leq -2I \quad\text{for $s\in\left[\frac12,1\right]$}.
$$
Integrating the above inequality between $s$ and $1$, and recalling \eqref{2603eq2}, we obtain
$$
F(s)\geq 2\,I\,\frac{1-s}{s}
$$  
which in particular implies that $F(s)>0$ for any $s\in[1/2,1)$.
\end{proof}

\begin{remark}
In order to give an estimate of $\bar \gamma$, we  mention that a tedious, but straightforward computation shows that if we compute $ \hat c $  and $c^\perp$ at $\gamma=2(1-s)$, we get
\begin{equation*}\label{daMarcello}
\hat c(2(1 - s)) = C_s\frac{1}{s(2s - 1)}, \quad c^\perp(2(1 - s)) = -C_s\frac{1}{s}.
\end{equation*}
Thus, when $k=2$ we have
$
c_2\left(2(1 - s)\right) = C_s \frac{2(1 - s)}{s(2s - 1)} > 0,
$
and therefore $\bar \gamma < 2(1 - s)$.
\end{remark}

\section{Liouville-type results for $\I_k^{\pm}$}
\label{secLiouville}

In this section we will prove a certain number of theorem of Liouville type i.e. of classifications of entire 
solutions or supersolutions that are bounded from below.

\subsection{Liouville results for superharmonic functions}
We state the results for $\I_k^+$. A dual result concerning $\I_k^-$ can be also given, but we omit the 
details. The computations in Proposition \ref{propbargamma} play a crucial role.
	\begin{teo}\label{linearLiouville}
		Consider the equation
		\begin{equation}\label{eqharmonic}
		\I_k^+ u = 0 \quad \mbox{in} \ \R^N.
		\end{equation}
		\begin{itemize}
			\item[(i)] { If $s \in [1/2, 1)$ and $k = 1$}, every viscosity supersolution $u$ to problem~\eqref{eqharmonic} which is bounded from below, is a constant.			
			\item[(ii)]  If $s \in (0,1/2)$ and $k = 1$, or $s \in (0,1)$ and $k \in \N$ with $2 \leq k \leq N$, then there exists nontrivial viscosity supersolution bounded from below to equation~\eqref{eqharmonic}.
		\end{itemize}
	\end{teo}
	\begin{proof}
	(i). By adding a constant, we can assume that $u \geq 0$. Consider first the case $s>\frac12$ and fix $\gamma \in (0, 2s - 1)$. Let $w_\gamma(x)=w_\gamma(|x|)=-|x|^\gamma$. 
	By Lemma \ref{lema2},  $\I_{\hat x}w_\gamma(x)\geq0$ for any $x\neq0$. In particular, we have that
		\begin{equation*}
		\I_1^+ w_\gamma(x) \geq 0 \quad \mbox{for} \ |x| \geq 1.
		\end{equation*}
		Thus, for every $R>1$ and denoting $\displaystyle m(1)=\min_{x\in \overline B_1}u(x)$, the function $\phi$ defined as 
		\begin{equation*}
		\phi(x) = m(1)\, \frac{w_\gamma(|x|) - w_\gamma(R)}{-w_\gamma(R)},	
		\end{equation*}
is a viscosity subsolution to~\eqref{eqharmonic} for $1<|x|<R$ and moreover  $\phi\leq u$ for $|x|\leq1$ and for 
$|x|\geq R$ . Then, by comparison principle, we have $\phi \leq u$ in $\R^N$. Thus, for each $|x|>1$ fixed, we 
let $R \to +\infty$ and then
		\begin{equation*}
		u(x) \geq m(1),
		\end{equation*}
from which we infer that $u \geq m(1)$ in $\R^N$. By the strong minimum principle, see  
Proposition~\ref{propSMP}, we conclude that $u$ is constant.	

In the case $s=\frac12$, let $w(x)=-\log(1+|x|)$. By  Theorem \ref{propPk}-(i) we have
\begin{equation}\label{0804eq1}
\I_1^+ w(x)=\I_{\hat x}w(x).
\end{equation}
A direct computation shows that 
\begin{equation}\label{0804eq2}
\I_{\hat x}w(x)=C_s\left(\frac{1}{1+|x|}\log(1+2|x|)+\log\frac{|x|^2}{1+2|x|}\right)
\end{equation}
and, if $x\geq1+\sqrt{2}$, it turns out that $\I_{\hat x}w(x)\geq0$. Thus, from \eqref{0804eq1}-\eqref{0804eq2},  we have
$$
\I^+_1w(x)\geq0 \quad \mbox{for} \ |x| \geq 1+\sqrt{2}.
$$
For any $R>1+\sqrt{2}$ set 
$$
\phi(r)=m(1+\sqrt{2})\frac{\log(1+R)-\log(1+|x|)}{\log(1+R)},
$$
where $m(1+\sqrt{2})=\displaystyle \min_{\overline B_{1+\sqrt{2}}}u$. As above, $\phi$ is a viscosity subsolution to~\eqref{eqharmonic} for $1+\sqrt{2}<|x|<R$ such that $\phi\leq u$ in $\mathbb R^N$. Sending $R\to+\infty$, we infer that $u$ is constant.

	\medskip
	
	\noindent
	(ii). Let $\bar \gamma$ be the constant given by Proposition~\ref{propbargamma} and consider for 
$0 < \gamma < \bar \gamma$ the function $u(x)=\min\{ 1, |x|^{-\gamma}\}$. By the basic principle saying that 
minima of supersolutions is supersolution, then $u$ is a nontrivial viscosity supersolution 
to~\eqref{eqharmonic} which is bounded from below.
		%
		%
		%
		%
		%
		%
	\end{proof}
Recall that $\bar \gamma \to 0$ and that $\I_2^+\to \PP_2^+$ as $s \to 1^-$. So even if Liouville type Theorems are valid for $\PP_2^+$ see Theorem 2.2 in~\cite{BGL}, the result of Theorem \ref{linearLiouville} are not in contradiction, since in a certain sense the solution construct here converges to the trivial solution.
\medskip


%
	

\subsection{Liouville-type result for the maximal operator $\I_k^+$}
In this subsection we assume $k=1$ and $s \in (0,1/2)$, or $k\geq2$ and $s\in(0,1)$. The aim is to prove the following
\begin{teo}\label{teoLiouville+}
	Let $\bar \gamma \in (0,1)$ be as in Proposition~\ref{propbargamma}. The equation 
	\begin{equation}\label{2004eq6}
	\I_k^+ u(x)+u^p(x)=0\quad\text{in $\R^N$}
	\end{equation}
	has nontrivial viscosity supersolutions if, and only if 
	$$
	p>1+\frac{2s}{\bar\gamma}.
	$$
\end{teo}
We divide the proof in several partial results. We start with the sufficient condition in the previous theorem.
\begin{prop}\label{LiouvillePlus}
	For any $p>1+\frac{2s}{\bar\gamma}$ there exist positive viscosity supersolutions of \eqref{2004eq6}.
\end{prop}
\begin{proof}
	For any $q\in\left[\frac{1}{p-1},\frac{\bar\gamma}{2s}\right)$, let 
	$$
	u(x)=\frac{1}{\left(1+|x|\right)^{2sq}}.
	$$
	As a consequence of Theorem \ref{propPk}, we have for any $x\in\R^N$ 
	\begin{equation*}
	\begin{split}
	&\I_k^+ u(x)=C_s\,\mathrm{P.V.} \int_{-\infty}^{+\infty}\left[{\left(1+||x|+\tau|\right)^{-2sq}}-{\left(1+|x|\right)^{-2sq}}\right]|\tau|^{-(1+2s)}\,d\tau \\
	&\quad+(k-1)C_s\int_{-\infty}^{+\infty}\left[{\left(1+\sqrt{|x|^2+\tau^2}\right)^{-2sq}}-{\left(1+|x|\right)^{-2sq}}\right]|\tau|^{-(1+2s)}\,d\tau\\
	&=\frac{1}{\left(1+|x|\right)^{2sq}}\left(C_s\,\mathrm{P.V.} \int_{-\infty}^{+\infty}\left[{\left(\frac{1}{1+|x|}+\left|\frac{|x|}{1+|x|}+\frac{\tau}{1+|x|}\right|\right)^{-2sq}}-1\right]|\tau|^{-(1+2s)}\,d\tau \right.\\
	&\left.\quad+(k-1)C_s\int_{-\infty}^{+\infty}\left[{\left(\frac{1}{1+|x|}+\sqrt{\frac{|x|^2}{(1+|x|)^2}+\frac{\tau^2}{(1+|x|)^2}}\,\right)^{-2sq}}-1\right]|\tau|^{-(1+2s)}\,d\tau\right).
	\end{split}
	\end{equation*}
	By the triangular inequality, for any $x\in\R^N$ and $\tau\in\R$,
	$$
	\frac{1}{1+|x|}+\left|\frac{|x|}{1+|x|}+\frac{\tau}{1+|x|}\right|\geq\left|1+\frac{\tau}{1+|x|}\right|
	$$
	and  a straightforward computation yields
	$$
	\left(\frac{1}{1+|x|}+\sqrt{\frac{|x|^2}{(1+|x|)^2}+\frac{\tau^2}{(1+|x|)^2}}\right)^2\geq1+\frac{\tau^2}{(1+|x|)^2}.
	$$
	Hence we obtain
	\begin{equation*}
	\begin{split}
	\I_k^+ u(x)&\leq\frac{1}{\left(1+|x|\right)^{2s(q+1)}}\left(C_s\,\mathrm{P.V.} \int_{-\infty}^{+\infty}\left[\left|1+\tau\right|^{-2sq}-1\right]|\tau|^{-(1+2s)}\,d\tau \right.\\
	&\left.\quad+(k-1)C_s\int_{-\infty}^{+\infty}\left[{\left(1+\tau^2\right)^{-sq}}-1\right]|\tau|^{-(1+2s)}\,d\tau\right)\\
	&=\frac{1}{\left(1+|x|\right)^{2s(q+1)}}\,c_k(2sq),
	\end{split}
	\end{equation*}
	where $c_k(\cdot)$ is the function defined in Proposition \ref{propbargamma}.
	Since $2sq<\bar\gamma$, then  $c_k(2sq)<0$. For $\varepsilon\in(0,(-c_k(2sq))^{1/(p-1)})$ and $v(x)=\varepsilon u(x)$ we conclude
	\begin{equation*}
	\begin{split}
	\I_k^- v(x)+v^p(x)&\leq\frac{\varepsilon}{\left(1+|x|\right)^{2s(q+1)}}\left(c_k(2sq)+\frac{\varepsilon^{p-1}}{\left(1+|x|\right)^{2s(qp-q-1)}}\right)\\
	&\leq\frac{\varepsilon}{\left(1+|x|\right)^{2s(q+1)}}\left(c_k(2sq)+\varepsilon^{p-1}\right)\leq0\,.
	\end{split}
	\end{equation*}
\end{proof}

For the necessary condition, we require some preliminary lemmas.
\begin{lema}\label{lem1}
	Let $\bar \gamma$ be as in Proposition~\ref{propbargamma}.  Given $r > 0$, we denote $\displaystyle m(r) = \min_{\overline B_r} u$.
	
	\begin{itemize}
		\item[$(i)$] If $u$ is a nonnegative viscosity supersolution of~\eqref{eqharmonic}, for any $\gamma>\bar\gamma$ there exists a positive constant $c=c(\gamma)$ such that 
		\begin{equation}\label{1804eq2}
		m(r)\geq c\,m(1)r^{-\gamma}\quad\forall r\geq1.
		\end{equation}
	\item[$(ii)$] If $u$ is  a positive supersolution of~\eqref{2004eq6} for some  $p<\frac{1+2s}{\bar\gamma}$, then 
there exists a positive constant $\bar c=\bar c(\bar\gamma,p,s,m(1))$ such that 
	\begin{equation}\label{2904eq2}
	m(r)\geq \bar c\,r^{-\bar\gamma}\quad\forall r\geq1.
	\end{equation}
	\end{itemize}
\end{lema}
\begin{proof}
	$(i)$ The statement \eqref{1804eq2} is trivial if $u\equiv0$. By the strong minimum principle, see Proposition \ref{propSMP}, we can then assume  $u>0$ in $\mathbb R^N$. 
	
	We claim that for $\varepsilon$ small enough (depending on $\gamma$) the function
	$$w(|x|)=
	\begin{cases}
	\varepsilon^{-\gamma} & \text{if $|x|\leq\varepsilon$}\\
	|x|^{-\gamma} & \text{if $|x|>\varepsilon$}
	\end{cases}
	$$
is a subsolution of $\I_k^+ u(x)=0$ for $|x|\geq	1$. Then \eqref{1804eq2} follows from the claim, since the 
function
	$$
	\phi(x)=m(1)\frac{w(|x|)-w(R)}{w(\varepsilon)-w(R)}
	$$
is, for any $R>1$, subsolution of $\I_k^+ u(x)=0$ for $1<|x|<R$. 
Moreover $u(x)\geq m(1)\geq\phi(x)$ for $|x|\leq1$ and $u(x)\geq0\geq\phi(x)$ if $|x|\geq R$. The comparison 
principle yields $u(x)\geq\phi(x)$ for $1<|x|<R$ and letting $R\to+\infty$ we infer that
	$$
	m(r)\geq m(1)\varepsilon^{\gamma}r^{-\gamma},
	$$
leading to \eqref{1804eq2} with $c=\varepsilon^{\gamma}$.
	
We  proceed with the proof of the claim. For $|x|\geq1$, we use that
	\begin{align*}
	\I_k^+ w(x) \geq & \ \I_{\hat{x}}w(x)+(k-1)\I_{x^\perp}w_\gamma(x) 
	\end{align*}
where $w_\gamma(x) = |x|^{-\gamma}$. Now we concentrate on $\I_{\hat x} w(x)$. For $|x| \geq 1$ we see that
\begin{align*}
\I_{\hat{x}}w(x) = \I_{\hat{x}} w_\gamma(x) - \int_{-|x|-\varepsilon}^{-|x|+\varepsilon}\frac{||x|+\tau|^{-\gamma}-\varepsilon^{-\gamma}}{|\tau|^{1+2s}}\,d\tau,
\end{align*}
from which, by Proposition~\ref{propbargamma} we conclude that
\begin{equation*}
\begin{split}
\I^+_kw(x) \geq & \  c_k(\gamma)|x|^{-(\gamma+2s)}-\int_{-|x|-\varepsilon}^{-|x|+\varepsilon}\frac{||x|+\tau|^{-\gamma}-\varepsilon^{-\gamma}}{|\tau|^{1+2s}}\,d\tau \\
= & \ |x|^{-(\gamma+2s)} \Big{(} c_k(\gamma) - \int_{-1-\frac{\varepsilon}{|x|}}^{-1+\frac{\varepsilon}{|x|}}\frac{|1+\tau|^{-\gamma}-\left(\frac{\varepsilon}{|x|}\right)^{-\gamma}}{|\tau|^{1+2s}}\,d\tau \Big{)}\,.
\end{split}
\end{equation*}

Let us denote $I$ the integral term in the right-hand side of the last inequality. Using that $|x| \geq 1$ and $\varepsilon < 1/2$ we have
\begin{equation*}
I \leq 2^{1+2s} \int_{-1-\varepsilon}^{-1+\varepsilon}|1+\tau|^{-\gamma}\,d\tau =  \frac{2^{2(1+s)}}{1-\gamma}\varepsilon^{1-\gamma} \leq \frac{16}{1-\gamma}\varepsilon^{1-\gamma}.
\end{equation*}


Using this, we conclude that
\begin{equation}\label{2904eq3}
\I^+_kw(x) \geq  \ |x|^{-(\gamma+2s)} (c_k(\gamma) - C \varepsilon^{1 - \gamma}),
\end{equation}
with $C = 16 (1 - \gamma)^{-1}$. Since $\gamma>\bar\gamma$ we have $c_k(\gamma) > 0$  and therefore it is sufficient to take $\varepsilon\leq\min\left\{\frac12,\left(\frac{c_k(\gamma)}{C}\right)^{\frac{1}{1-\gamma}}\right\}$ to conclude the proof of the claim.
	
	\medskip
	\noindent
	$(ii)$ Let us consider 
	$$w(|x|)=
	\begin{cases}
	\varepsilon^{-\bar\gamma} & \text{if $|x|\leq\varepsilon$}\\
	|x|^{-\bar\gamma} & \text{if $|x|>\varepsilon$}.
	\end{cases}
	$$
	
Similarly to \eqref{2904eq3}, using the fact that $c_k(\bar\gamma)=0$, we have for $|x|\geq1$
	$$
	\I_k^+ w(x)\geq-\int_{-|x|-\varepsilon}^{-|x|+\varepsilon}\frac{||x|+\tau|^{-\bar\gamma}-\varepsilon^{-\bar\gamma}}{|\tau|^{1+2s}}\,d\tau.
	$$
Assuming  $\varepsilon\leq\frac12$, we infer that 
	\begin{equation}\label{2904eq4}
	\begin{split}
	\I_k^+ w(x)&\geq-\frac{2^{1+2s}}{|x|^{1+2s}}\int_{-|x|-\varepsilon}^{-|x|+\varepsilon}{||x|+\tau|^{-\bar\gamma}-\varepsilon^{-\bar\gamma}}\,d\tau=-4^{1+s}\frac{\bar\gamma}{1-\bar\gamma}\varepsilon^{1-\bar\gamma}\frac{1}{|x|^{1+2s}}.
	\end{split}
	\end{equation}
For any $R\geq2^{\frac{1-\bar\gamma}{\bar\gamma}}$, the function 
	$$
	\phi(x)=m(1)\frac{w(|x|)-w(R)}{w(\varepsilon)-w(R)}
	$$
satisfies, for $|x|\geq1$, the inequality
	\begin{equation}\label{2904eq5}
	\I_k^+\phi(x)\geq-\tilde c\,\varepsilon^{1-\bar\gamma}\frac{1}{|x|^{1+2s}},
	\end{equation}
with $\tilde c=m(1)2^{3+2s-\bar\gamma}\frac{\bar\gamma}{1-\bar\gamma}$.\\
	Now we apply \eqref{1804eq2} with $\gamma=\frac{1+2s}{p}$. Note that $\gamma>\bar\gamma$ by the assumption $p<\frac{1+2s}{\bar\gamma}$. From \eqref{2004eq6} we then obtain
	\begin{equation}\label{2904eq6}
	\I_k^+u(x)\leq -u^p(x)\leq-(cm(1))^p\frac{1}{|x|^{1+2s}},
	\end{equation}
where $c$ is the constant appearing in \eqref{1804eq2}. Now, from \eqref{2904eq5}-\eqref{2904eq6}, taking $\varepsilon=\varepsilon(\bar\gamma,p,s,m(1))$ small enough, we have
	$$
	\I_k^+u(x)\leq\I_k^+\phi(x)\quad\forall |x|\geq1.
	$$ 
Since $u\geq\phi$ for $|x|\leq1$ and $|x|\geq R$, by comparison principle $u\geq \phi$ for $|x|\in[1,R]$. 
Sending $R\to+\infty$, we obtain 
	$$
	m(r)\geq m(1)\varepsilon^{\bar\gamma}r^{-\bar\gamma},
	$$
which is exactly \eqref{2904eq2} with $\bar c=m(1)\varepsilon^{\bar\gamma}$.
\end{proof}

\begin{lema}\label{lem2}
	Let $\bar \gamma$ be as in Proposition~\ref{propbargamma}. Let $u$ be a nonnegative viscosity 
supersolution of~\eqref{eqharmonic}. Then, 
for any $\gamma\geq\bar\gamma$ there exists a positive constant $c=c(\gamma)$ such that 
	\begin{equation}\label{2004eq2}
	m(R)\geq c\,m\left(\frac R2\right)\quad\forall R>0.
	\end{equation}
\end{lema}
\begin{proof}
Let $\gamma\geq\bar\gamma$, $R>0$ and $R_0= \varepsilon R$ for some $\varepsilon \in (0, 1/4)$ to be fixed. 
	Consider the function 
	$$w_R(|x|)=
	\begin{cases}
	R_0^{-\gamma} & \text{if $|x|\leq R_0$}\\
	|x|^{-\gamma} & \text{if $R_0<|x|$}.
	\end{cases}
	$$
	
	We claim that the function $\phi(x) = (w_R(|x|) - (2R)^{-\gamma})_+$ satisfies
	$$
	\I_k^+ \phi \geq 0 \quad \mbox{in} \ B_{2R} \setminus B_{R/2}.
	$$
	
	Assuming the claim is true, the function
	$$
	\tilde \phi(x) = m\left(\frac R2\right)\frac{\phi(x)}{R_0^{-\gamma} - (2R)^{-\gamma}},
	$$
	solves $\I_k^+ \tilde \phi(x)\geq 0$ for $|x|\in\left(\frac R2,2R\right)$. Since $u(x)\geq m\left(\frac R2\right)\geq \tilde \phi(x)$ for $|x|\leq\frac R2$ and $u(x)\geq 0 = \tilde \phi(x)$ if $|x|\geq2R$, by comparison principle we get $u(x)\geq \tilde \phi(x)$ for $|x|\in\left(\frac R2,2R\right)$. In particular, we have $m(R) \geq \min_{\overline B_R} \tilde \phi$ from which we obtain
	$$
	m(R)\geq m\left(\frac R2\right)\frac{R^{-\gamma}-(2R)^{-\gamma}}{R_0^{-\gamma}-(2R)^{-\gamma}}= c \  m\left(\frac R2\right) ,
	$$
	where $c = c(\varepsilon, \gamma) := \frac{1 - 2^{-\gamma}}{\varepsilon^{-\gamma} - 2^{-\gamma}}$. 
	Then \eqref{2004eq2} holds with this constant $c$.
	
	Now we prove the claim. By definition, for each $|x|\in\left(\frac R2,2R\right)$ we have
	\begin{align}\label{cuore}
	\I_k^+ \phi(x) & \geq \I_{\hat{x}}\phi(x)+(k-1)\I_{x^\perp}\phi(x).
	\end{align}
	
	As in Proposition~\ref{propbargamma} we denote $w_\gamma(x)=|x|^{-\gamma}$. Denoting $A = \{ \tau \in \R : |x|^2 + \tau^2 \leq (2R)^2\}$ we have $\phi(x + \tau x^\perp) = w_\gamma(x + \tau x^\perp) - (2R)^{-\gamma}$ for $\tau \in A$, while for $\tau \in A^c$ it holds that $\phi(x+\tau x^\perp)=0$ and $w_\gamma(x + \tau x^\perp) \leq (2R)^{-\gamma}$.
	Then we have 
	\begin{equation}\label{cuore1}
	\begin{split}
	\I_{x^\perp}\phi(x) = & \ C_s \int_{A} [w_\gamma(x + \tau x^\perp) - w_\gamma(x)]|\tau|^{-(1 + 2s)}d\tau \\
	& \ + C_s \int_{\R \setminus A} [(2R)^{-\gamma} - w_\gamma(x)]|\tau|^{-(1 + 2s)}d\tau \\
	\geq & \ \I_{x^\perp} w_\gamma(x)
	\end{split}
	\end{equation}
	
	We employ a similar argument for $\I_{\hat{x}} \phi(x)$. This time we denote the (disjoint) sets 
	$$
	A = [-R_0 - |x|, R_0 - |x|], \quad B = \{ \tau \in \R : ||x| + \tau| \geq 2R \}.
	$$ 
	
	Thus, by definition we have
	\begin{align*}
	\I_{\hat x} \phi(x) = & \I_{\hat x} w_\gamma(x)  + C_s \int_A [R_0^{-\gamma} - w_\gamma(x + \tau \hat x)] |\tau|^{-(1 + 2s)}d\tau \\
 	& + C_s \int_B [(2R)^{-\gamma} - w_\gamma(x + \tau \hat x)] |\tau|^{-(1 + 2s)} d\tau \\
 	=: & \ \I_{\hat x} w_\gamma(x) + I_1 + I_2.
	\end{align*}
	
	For $I_1$, notice that $|\tau| \geq R/4$ for each $\tau \in A$. Then, we have
	\begin{equation*}
	\begin{split}
	I_1 \geq -C_s \int_{A} ||x| + \tau|^{-\gamma} |\tau|^{-(1 + 2s)}d \tau \geq -C_s\left(\frac4R\right)^{1+2s}\int_{-R_0}^{R_0}|\tau|^{-\gamma}\,d\tau 
	\end{split}
	\end{equation*}
	from which, by the choice of $R_0$ we conclude
	\begin{equation}\label{2004eq4}
	I_1 \geq - c_1 R^{-\gamma-2s},
	\end{equation}
where $c_1 = C_s 2^{3+4s}\frac{\varepsilon^{1 - \gamma}}{1-\gamma}$. Observe that this constant tends to zero as $\varepsilon \to 0$. 
	
	As far as $I_2$ is concerned, notice that the integrand is nonnegative. Thus, if we denote $B' = \{ \tau \in \R : ||x| + \tau| \geq 3R \} \subset B$ we have
	\begin{equation*}
	I_2 \geq
	C_s (2^{-\gamma} - 3^{-\gamma}) R^{-\gamma} \int_{B'}  |\tau|^{-(1 + 2s)} d\tau  \geq 
	C_s (2^{-\gamma} - 3^{-\gamma}) R^{-\gamma} \int_{5R}^{+\infty}  \tau^{-(1 + 2s)} d\tau, 
	\end{equation*}
	from which we get
	\begin{equation}\label{2004eq5}
	I_2 \geq c_2 R^{-(\gamma + 2s)},
	\end{equation}
	with $c_2 = C_s (2^{-\gamma} - 3^{-\gamma}) \frac{5^{-2s}}{2s}$. Observe that this constant is independent of $\varepsilon$. 
	
	Putting together \eqref{2004eq4}-\eqref{2004eq5} into the expression of $\I_{\hat x} \phi(x)$ above, we conclude that for $\varepsilon$ small enough we get
	\begin{equation*}
	\I_{\hat x} \phi(x) \geq \I_{\hat x} w_\gamma(x),
	\end{equation*}
	and from here, replacing this and~\eqref{cuore1} into~\eqref{cuore}, we conclude the claim. The proof is now complete.
%
%
\end{proof}

\begin{lema}\label{2904lem1}
	There exists a positive constant $c=c(k,s)$ such that the function $$\Gamma(x)=\frac{\ln|x|}{|x|^{\bar\gamma}}\,, \quad x\neq0$$ satisfies  
	\begin{equation}\label{2904eq7}
	\I_k^+\Gamma(x)\geq-\frac{c}{|x|^{\bar\gamma+2s}}\,,\quad x\neq0.
	\end{equation}
\end{lema}
\begin{proof}
	Let $w_{\bar\gamma}(x)=|x|^{-\bar\gamma}$. For $x\neq0$ we have   
	\begin{equation*}
	\begin{split}
	\I_k^+\Gamma(x)&\geq\I_{\hat{x}}\Gamma(x)+(k-1)\I_{x^\perp}\Gamma(x)\\
	&=\ln|x|\,\I_k^+ w_{\bar\gamma}(x)+\frac{1}{|x|^{\bar\gamma+2s}}\left[C_s\mathrm{P.V.}\int_{-\infty}^{+\infty}\frac{\ln|1+\tau|}{|1+\tau|^{\bar\gamma}\,|\tau|^{1+2s}}\,d\tau\right.\\&\qquad\left.+C_s(k-1)\int_0^{+\infty}\frac{\ln(1+\tau^2)}{(1+\tau^2)^{\bar\gamma/2}\,\tau^{1+2s}}\,d\tau\right].
	\end{split}
	\end{equation*}
	
	Since, by Proposition \ref{propbargamma}, $\bar w(x)$ solves $\I_k^+ w_{\bar\gamma}(x)=0$ for $x\neq0$, then \eqref{2904eq7} follows. 
\end{proof}

\medskip

Now we are in position to provide the

\noindent
{\bf \textit{Proof of Theorem~\ref{teoLiouville+}}:}
	The existence of nontrivial supersolutions of \eqref{2004eq6} when $p>1+\frac{2s}{\bar\gamma}$ is a consequence of Proposition \ref{LiouvillePlus}.
	
	Let $p\leq1+\frac{2s}{\bar\gamma}$. We shall prove that $u\equiv0$ is the only nonnegative supersolution of \eqref{2004eq6}. By contradiction we suppose the contrary. Let $u$ be a nontrivial supersolution of \eqref{2004eq6}. By the strong minimum principle, see Remark \ref{SMP}, $u>0$  in $\mathbb R^N$.\\ 
	Let $\eta(|x|)$ be a cut-off function such that $\eta(|x|)=0$ for $|x|\geq1$ and $\eta(|x|)=1$ for $|x|\leq\frac12$. Define $\xi(x)=m\left(\frac R2\right)\eta\left(\frac{|x|}{R}\right)$. Since $\I_k^+\eta(x)\geq- C_\eta$, for some positive constant $C_\eta$, by scaling it turns out that
	\begin{equation}\label{2004eq7}
	\I_k^+\xi(x)\geq-\frac{C_\eta m\left(\frac R2\right)}{R^{2s}}.
	\end{equation} 
	Moreover $u(x)\geq\xi(x)$ for $|x|\in[0,\frac R2]\cup[R,+\infty]$ and $u(x)=\xi(x)$ for some $|x|=\frac R2$. Then there exists $x_R\in\mathbb R^N$ such that $|x_R|\in[\frac R2,R)$ and  $u(x)-\xi(x)\geq u(x_R)-\xi(x_R)$ for any $x\in\R^N$. Then $\I_k^+\xi(x_R)+u^p(x_R)\leq0$ and by \eqref{2004eq7} we infer that 	
	\begin{equation*}
	m^p(R)\leq u^p(x_R)\leq-\I_k^+\xi(x_R)\leq\frac{C_\eta m\left(\frac R2\right)}{R^{2s}}.
	\end{equation*}
	Then, using \eqref{2004eq2}, we have 
	\begin{equation}\label{2004eq8}
	m^{p-1}(R)\leq\frac{C}{R^{2s}}
	\end{equation}
	for a positive constant $C$. \\If $p\leq1$ then 
		\begin{equation}\label{2004eq8'}
	m^{p-1}(1)\leq\frac{C}{R^{2s}}
	\end{equation}
		for any $R\geq1$. Letting $R\to+\infty$ in \eqref{2004eq8'}, we infer  that $m(1)=0$. This violates the positivity of $u$. In what follows, the case $p>1$ is considered.
	
	\noindent
	\smallskip
	\textbf{1.- Case $p<1+\frac{2s}{\bar\gamma}$.}  Let $\gamma>\bar\gamma$ be such that 
	\begin{equation}\label{2004eq9}
	\frac{2s}{p-1}-\gamma>0.
	\end{equation}
	From \eqref{1804eq2} and \eqref{2004eq8} we have 
	$$
	m(1)\leq\frac{C}{R^{\frac{2s}{p-1}-\gamma}},
	$$
	for a positive constant $C$. Sending $R\to+\infty$, and using \eqref{2004eq9}, we again reach the contradiction that $m(1)=0$.
	
	\medskip 
	
	\noindent
	\textbf{2.- Case $p=1+\frac{2s}{\bar\gamma}$.} By contradiction let $u$ be a positive supersolution of \eqref{2004eq6}. From \eqref{2004eq8} we have the bound
	\begin{equation}\label{2904eq8}
	m(R)R^{\bar\gamma}\leq C,
	\end{equation}
	for some $C>0$. For $x\neq0$, let $\Gamma(|x|)=\frac{\ln|x|}{|x|^{\bar\gamma}}$. We have $\Gamma\left(e^{1/{\bar\gamma}}\right)=\max_{|x|>0}\Gamma(|x|)$ and, by Lemma \ref{2904lem1},  
	\begin{equation}\label{2904eq9}
	\I_k^+\Gamma\geq-\frac{c}{|x|^{\bar\gamma+2s}}\qquad\text{for $x\neq 0$}.
	\end{equation}
	Consider now, for $r_2>r_1>e^{1/{\bar\gamma}}$, the comparison function 
	$$
	\phi(x)=m(r_1)\frac{\Gamma(|x|)-\Gamma(r_2)}{\Gamma(e^{1/{\bar\gamma}})-\Gamma(r_2)},
	$$
	which, by construction, satisfies $\phi(x)\leq u(x)$ for $|x|\leq r_1$ and $|x|\geq r_2$. Moreover, by \eqref{2904eq9},
	$$
	\I_k^+\phi(x)\geq-\frac{cm(r_1)}{\Gamma(e^{1/{\bar\gamma}})-\Gamma(r_2)}\frac{1}{|x|^{\bar\gamma+2s}}\qquad\text{for $x\neq0$}.
	$$
	For $r_2$ sufficiently large we may further assume that $\Gamma(e^{1/{\bar\gamma}})-\Gamma(r_2)\geq\frac12\Gamma(e^{1/{\bar\gamma}})$, so that 
	\begin{equation}\label{2904eq10}
	\I_k^+\phi(x)\geq-\frac{2cm(r_1)}{\Gamma(e^{1/{\bar\gamma}})}\frac{1}{|x|^{\bar\gamma+2s}}.
	\end{equation}
	By Lemma \ref{lem1} and \eqref{2904eq2}, we also have  
	\begin{equation}\label{2904eq11}
	\I_k^+u(x)\leq -(u(x))^{1+\frac{2s}{\bar\gamma}}\leq-(m(|x|))^{1+\frac{2s}{\bar\gamma}}\leq -(\bar c)^{1+\frac{2s}{\bar\gamma}}\frac{1}{|x|^{\bar\gamma+2s}}.
	\end{equation}
	Since $m(r_1)\to0$ as $r_1\to+\infty$, in view of \eqref{2904eq8}, we can fix $r_1$ large enough and use \eqref{2904eq10}-\eqref{2904eq11} to obtain that $\I_k^+u(x)\leq\I_k^+\phi(x)$ for any $|x|\in(r_1,r_2)$. Hence, by comparison, $u(x)\geq\phi(x)$ and passing to the limit as $r_2\to+\infty$ we deduce that 
	$$
	m(r)r^{\bar\gamma}\geq\frac{m(r_1)}{\Gamma(e^{1/{\bar\gamma}})}\ln r\quad \forall r>r_1,
	$$
	which  is in contradiction to \eqref{2904eq8}. \qed


\subsection{Liouville-type result for the minimal operator $\I_k^-$ with $k<N$}
When $k < N$, we infer from Theorem \ref{propPk}-(ii) (see also Remark \ref{rmkg}) that for any smooth bounded radial function $u(x)=\tilde g(|x|^2)$ such that  $\tilde g(r)$ is convex for $r\geq0$, one has
\begin{equation}\label{eq2}
\I_k^- u(x)=k\,\I_{x^\perp}u(x),
\end{equation}
$x^\perp$ being any unit vector orthogonal to $x$.

This is the key fact to conclude the following theorem
\begin{teo}\label{LiouvilleMinus}
Let $s \in (0,1)$, $1 \leq k < N$ and consider the equation
\begin{equation}\label{eq1}
\I_k^- u(x)+u^p(x)=0\quad\text{in $\R^N$.}
\end{equation} Then
\begin{itemize}
			\item[(i)] for any $p\geq1$ there exist positive classical solutions of \eqref{eq1};
			\item[(ii)] for any $p\in(0,1)$ there exist nonnegative viscosity solutions $u\not\equiv0$ of \eqref{eq1}.
\end{itemize}
\end{teo}

\begin{proof}
(i). We first consider the case $p>1$. For $r\geq0$, let 
\begin{equation*}
\tilde g(r)=\frac{\alpha}{(1+r)^{\frac{s}{p-1}}}.
\end{equation*}
We claim that for a suitable choice of $\alpha=\alpha(k,s,p)>0$ the function $u(x)=\tilde g(|x|^2)$ is solution of \eqref{eq1}. Since 
\begin{align*}
\I_{x^\perp} u(x) = & 2C_s \alpha \int_{0}^{+\infty}\left[(1+|x|^2+\tau^2)^{-\frac{s}{p-1}}-(1+|x|^2)^{-\frac{s}{p-1}}\right]\tau^{-(1+2s)}\,d\tau \\
= & 2C_s \alpha (1+|x|^2)^{-\frac{s}{p-1}} \int_{0}^{+\infty}\left[\Big{(} 1+ \Big{(} \frac{\tau}{\sqrt{1+|x|^2}} \Big{)}^2 \Big{)}^{-\frac{s}{p-1}}-1\right]\tau^{-(1+2s)}\,d\tau \\
= & 2C_s \alpha (1+|x|^2)^{-\frac{sp}{p-1}} \int_{0}^{+\infty}\left[(1 + \tau^2 )^{-\frac{s}{p-1}}-1\right]\tau^{-(1+2s)}\,d\tau\,,
\end{align*}
we obtain from \eqref{eq2}   that 
$$
\I_k^- u(x) = k \I_{x^\perp} u(x) = -\alpha \bar c (1 + |x|^2)^{-\frac{sp}{p-1}}
$$ 
where $\bar c=2C_sk\int_{0}^{+\infty}\left[1-(1 + \tau^2 )^{-\frac{s}{p-1}}\right]\tau^{-(1+2s)}\,d\tau > 0$. Hence, we get that
\begin{equation*}
\I_k^- u(x) + u^p (x)= (1 + |x|^2)^{-\frac{sp}{p-1}} \Big{(} -\alpha \bar c + \alpha^p \Big{)},
\end{equation*}
from which, taking $\alpha = \bar c^{1/(p-1)}$ we conclude the result.
Moreover, by scaling, it turns out that for any  $a \neq 0$, the function 
$$
u(x)=\frac{\alpha}{(a^2+|x|^2)^\frac{s}{p-1}}
$$ 
is again solution to \eqref{eq2} for the same choice of $\alpha$ ad before.

In the case $p=1$ we follow a similar argument with a different radial profile. More specifically, for $\beta>0$ to be fixed, we consider the function 
$$
\tilde g(r)=e^{-\beta r}.
$$ 
As above, for $u(x)=\tilde g(|x|^2)$, we have that
$\I_k^- u(x) = k \I_{x^\perp} u(x)$ . It is easy to see that
\begin{equation*}
\I_{x^\perp} u(x) = -e^{-\beta |x|^2} F(\beta),
\end{equation*}
where
\begin{equation*}
F(\beta)= 2C_s\int_{0}^{+\infty}\left(1-e^{-\beta \tau^2}\right)\tau^{-(1+2s)}\,d\tau > 0.
\end{equation*}

Thus, we see that
\begin{equation*}
\I_k^- u(x) + u(x) = -e^{-\beta |x|^2} (k F(\beta) - 1).
\end{equation*}

By Fatou's lemma, one has
$$
+\infty=\int_0^{+\infty}\tau^{-(1+2s)}\,d\tau\leq\liminf_{\beta\to+\infty}\int_{0}^{+\infty}\left(1-e^{-\beta \tau^2}\right)\tau^{-(1+2s)}\,d\tau,
$$
from which we conclude that
\begin{equation}\label{eq3}
\lim_{\beta\to+\infty}F(\beta)=+\infty.
\end{equation}

Moreover, for $\beta\in(0,1]$, $$\left(1-e^{-\beta \tau^2}\right)\tau^{-(1+2s)}\leq\min\left\{\frac{1}{\tau^{2s-1}},\frac{1}{\tau^{2s+1}}\right\}\in L^1((0,+\infty))$$
and by Lebesgue's Theorem we infer that
\begin{equation}\label{eq4}
\lim_{\beta\to0^+}F(\beta)=0.
\end{equation} 
Since $F(\beta)$ is continuous (again  by Lebesgue's Theorem) we infer, by \eqref{eq3}-\eqref{eq4},  that there exists $\bar\beta>0$ such that $F(\bar\beta)=\frac{1}{k}$. Then $u(x)=e^{-\bar\beta|x|^2}$ is solution of \eqref{eq1} with $p=1$. We conclude observing that, by homogeneity, for any $b>0$ the function $u(x)=be^{-\bar\beta|x|^2}$ is still a positive entire solution of \eqref{eq1}.

\noindent
(ii). We shall prove that  radial function
$$
u(x)=\alpha(R^2-|x|^2)_+^{\frac{s}{1-p}}
$$
is, for a suitable choice of $\alpha>0$ and for any $R>0$, a viscosity solution of the equation
\begin{equation}\label{2903eq1}
k\I_{x^\perp} u(x)+u^p(x)=0\quad\text{in $\R^N$.}
\end{equation}
Note that the representation formula  \eqref{eq2} is missing in this case, since the function $u$  does not fit in the assumption of Theorem  \ref{propPk}-(ii). Nevertheless using the inequality $$
\I_k^- u(x)\leq k \I_{x^\perp} u(x),
$$
which holds for any admissible function $u$, just using the minimality of the operator $\I_k^-$ among the family of $k$-dimensional orthonormal subsets of $\mathbb R^N$, we can still  to conclude that $u$ is a viscosity supersolution of \eqref{eq1}.\\
To show \eqref{2903eq1} we first note that $u$ is a smooth function  for $|x|\neq R$. In particular we have
\begin{equation}\label{2903eq2}
\I_{x^\perp} u(x)=0=u^p(x)\quad\text{if $|x|>R$,}
\end{equation}
while if $|x|<R$ 
$$\I_{x^\perp} u(x)=-\alpha\bar c(R^2-|x|^2)^{\frac{ps}{1-p}}\,,$$
where $\bar c=2C_sk\int_0^{+\infty}\left(1-(1-\tau^2)_+^{\frac{s}{1-p}}\right)\tau^{-(1+2s)}\,d\tau>0$.
Hence
\begin{equation}\label{2903eq3}
\I_k^- u(x) + u^p(x) =(R^2-|x|^2)^{\frac{ps}{1-p}}\left(-\alpha\bar c+\alpha^p\right) \quad\text{if $|x|<R$.}
\end{equation}
 Choosing 
$\alpha=\bar c^{-\frac{1}{1-p}}$, we infer from \eqref{2903eq2}-\eqref{2903eq3} that  $u$  satisfies, in the classical sense, the equation \eqref{2903eq1} for any $|x|\neq R$.  \\
It remains to prove \eqref{2903eq1} whenever $|x|=R$.  It is plain that $u$ is a viscosity subsolution since there are no test functions touching $u$ from above at $x$ when $|x|=R$. In order to prove that $u$ is a supersolution of \eqref{2903eq1}, let $x_0\in\mathbb R^N$ such that $|x_0|=R$ and let $\varphi\in C^2(B_\delta(x_0))$ such that
\begin{equation}\label{2903eq4}
(u-\varphi)(x_0)=0\leq (u-\varphi)(x)\quad\forall x\in B_\delta(x_0)
\end{equation}
for some positive $\delta$. Then
\begin{equation}\label{2903eq5}
\begin{split}
\I_k^- (u,\varphi,x_0,\delta)+ u^p(x_0)&=\min_{ \xi \in \mathcal V_k} \{ \sum_{i=1}^{k} \I_{\xi_i, \delta} \varphi(x_0) + \I_{\xi_i}^\delta u(x_0)\}\\
&\leq k\I_{x^\perp,\delta}\varphi(x_0)+k\I^\delta_{x^\perp}u(x_0).
\end{split}
\end{equation}
Since $u(x_0+\tau x^\perp)=u(x_0)=0$ for any $\tau\in\mathbb R$, it is clear that $\I^\delta_{x^\perp}u(x_0)=0$. On the other hand, from \eqref{2903eq4} we infer that 
$$
\varphi(x_0+\tau x^\perp)\leq u(x_0+\tau x^\perp)=\varphi(x_0)\quad\forall \tau\in(-\delta,\delta)
$$ 
and
$$
\I_{x^\perp,\delta}\varphi(x_0)=C_s \mathrm{P.V.}\int_{-\delta}^\delta[\varphi(x_0+\tau x^\perp)-\varphi(x_0)]|\tau|^{-(1+2s)}\leq0.
$$
Hence by \eqref{2903eq5} we conclude that $$\I_k^- (u,\varphi,x_0,\delta)+ u^p(x_0)\leq0$$
as we wanted to show.
\end{proof}


\subsection{Liouville-type theorem for the minimal operator $\I_N^-$}
We start with the critical exponent associated to this operator. Let us remember that, by Theorem \ref{propPk}-(ii), the minimal operator  $\I^-_Nu$ coincides, within  a suitable class of radial function including as the main example the function $u(x)=|x|^{-{\gamma}}$,   with $N\I_{\xi^*}u$. Then a fundamental solution for the integral operator $\I_{\xi^*}$ is in turn a fundamental solution for $\I^-_N$.

\begin{lema}\label{preliminary lemma}
For  $s \in (0,1)$ and $\gamma > 0$, let 
\begin{equation}\label{eqc}
c(\gamma):=\int_0^{+\infty}\frac{\left(1+\tau^2+\frac{2}{\sqrt{N}}\tau\right)^{-\gamma/2}+\left(1+\tau^2-\frac{2}{\sqrt{N}}\tau\right)^{-\gamma/2}-2}{\tau^{1+2s}}\,d\tau\,.
\end{equation}

Then, there exists  a unique $ \tilde\gamma=\tilde\gamma(N,s) > 0$ such that $c(\gamma)<0$ for $\gamma<\tilde\gamma$, $c(\tilde\gamma)=0$ and $c(\gamma)>0$ for $\gamma>\tilde\gamma$.
\end{lema}
\begin{proof}
By Lebesgue's theorem we easily infer that $c(\gamma)\to0$ as $\gamma\to0^+$ and that
\begin{equation}\label{eq1july6}
c'(0)=-\frac12\int_0^{+\infty}\frac{\ln\left(1+2\left(1-\frac2N\right)\tau^2+\tau^4\right)}{\tau^{1+2s}}<0.
\end{equation}
Moreover, for any $\gamma>0$, we have
$$
c''(\gamma)=\int_{0}^{+\infty}\frac{f(\tau)+f(-\tau)}{\tau^{1+2s}}\,d\tau
$$ 
where $f(\tau)=\left(1+\tau^2+\frac{2}{\sqrt{N}}\tau\right)^{-\gamma/2}\ln^2\left(1+\tau^2+\frac{2}{\sqrt{N}}\tau\right)$. Since $f(\tau)\geq0$ for any $\tau$, then $c(\gamma)$ is  convex in $[0,+\infty)$. We claim that
\begin{equation}\label{eq2july6}
\lim_{\gamma\to+\infty}c(\gamma)=+\infty.
\end{equation}
Then, using \eqref{eq1july6}-\eqref{eq2july6}, we deduce that there exists  $\tilde\gamma=\tilde\gamma(N,s)>0$ such that $c(\tilde\gamma)=0$, $c(\gamma)<0$ for $\gamma<\tilde\gamma$ and $c(\gamma)>0$ for $\gamma>\tilde\gamma$.

To show \eqref{eq2july6} let 
$$g(\tau)=\left(1+\tau^2+\frac{2}{\sqrt{N}}\tau\right)^{-\gamma/2}+\left(1+\tau^2-\frac{2}{\sqrt{N}}\tau\right)^{-\gamma/2}-2,
$$
so that 
\begin{equation}
\begin{split}
c(\gamma)&=\int_0^{+\infty}\frac{g(\tau)}{\tau^{1+2s}}\,d\tau\\
&=\int_0^{\frac{1}{4\sqrt{N}}}\frac{g(\tau)}{\tau^{1+2s}}\,d\tau+\int_{\frac{1}{4\sqrt{N}}}^{\frac{1}{2\sqrt{N}}}\frac{g(\tau)}{\tau^{1+2s}}\,d\tau+\int_{\frac{1}{2\sqrt{N}}}^{+\infty}\frac{g(\tau)}{\tau^{1+2s}}\,d\tau\\
&=:I_1+I_2+I_3.
\end{split}
\end{equation}
We shall prove that $I_1$ and $I_3$ are bounded from below, while $I_2\to+\infty$ as $\gamma\to+\infty$. \\
Since $g(\tau)\geq-2$ for any $\tau>0$, we have
$$
I_3\geq-2\int_{\frac{1}{2\sqrt{N}}}^{+\infty}\frac{1}{\tau^{1+2s}}\,d\tau=-\frac{(2\sqrt{N})^{2s}}{s}.
$$
Moreover
\begin{equation*}
\begin{split}
g''(\tau)&=\gamma\left(1+\tau^2+\frac{2}{\sqrt{N}}\tau\right)^{-\gamma/2-2}\left((\gamma+2)(\tau+\frac{1}{\sqrt{N}})^2-1-\tau^2-\frac{2}{\sqrt{N}}\tau\right)\\
&\quad+\gamma\left(1+\tau^2-\frac{2}{\sqrt{N}}\tau\right)^{-\gamma/2-2}\left((\gamma+2)(\tau-\frac{1}{\sqrt{N}})^2-1-\tau^2+\frac{2}{\sqrt{N}}\tau\right).
\end{split}
\end{equation*}
Then, for $\gamma$ sufficiently large, $g(\tau)$ is convex in $[0,\frac{1}{4\sqrt{N}}]$. Since $g'(0)=0$ we infer that $g(\tau)\geq0$ for any $\tau\in[0,\frac{1}{4\sqrt{N}}]$. Hence $I_1\geq0$.\\
For $\tau\in[\frac{1}{4\sqrt{N}},\frac{1}{2\sqrt{N}}]$
$$
g(\tau)\geq\left(1+\tau^2-\frac{2}{\sqrt{N}}\tau\right)^{-\gamma/2}-2\geq\left(1-\frac{7}{16N}\right)^{-\gamma/2}-2
$$
and 
$$
I_2\geq\left(\left(1-\frac{7}{16N}\right)^{-\gamma/2}-2\right)\int_{\frac{1}{4\sqrt{N}}}^{\frac{1}{2\sqrt{N}}}\frac{1}{\tau^{1+2s}}\,d\tau\to+\infty\quad\text{as $\gamma\to+\infty$}.
$$
 \end{proof}
\begin{remark}\label{rmktildegamma}
If $N\geq3$ the value $\tilde\gamma$ in Lemma \ref{preliminary lemma} is in fact strictly larger than 1. This is a consequence of the fact that the function 
$$
f(\tau)=\left(1+\tau^2+\frac{2}{\sqrt{N}}\tau\right)^{-1/2}+\left(1+\tau^2-\frac{2}{\sqrt{N}}\tau\right)^{-1/2}-2
$$
is negative for any $\tau>0$, i.e. $c(1)<0$, which, together with the convexity of $c(\gamma)$, leads to $c(\gamma)<0$ for any $\gamma\in(0,1]$. 
\end{remark}

\medskip

The main result of this subsection is the following
\begin{teo}[\textbf{Liouville}]\label{Liouville-N}
	The equation 
	\begin{equation}\label{2407eq1}
	\I_N^-u(x)+u^p(x)=0\quad\text{in $\R^N$}
	\end{equation}
	has nontrivial viscosity supersolutions if, and only if, $p>1+\frac{2s}{\tilde\gamma}$.
\end{teo}

As before, we divide the proof of the previous theorem in several partial results. We start with the
\begin{prop}\label{propexi}
For any $p>1+\frac{2s}{\tilde\gamma}$ there exist positive viscosity supersolutions of the equation
$$
\I^-_Nu(x)+u^p(x)=0\quad\text{in $\R^N$.}
$$
\end{prop}
\begin{proof}
For $q\in\left[\frac{1}{p-1},\frac{\tilde\gamma}{2s}\right)$ we consider the function
$$
u(x)=\frac{1}{(1+|x|)^{2sq}}\,.
$$
Using Theorem \ref{propPk}, see also Remark \ref{rmkg},  for any fixed $x\in\mathbb R^N$, $x\neq0$, it holds
$$
\I^-_Nu(x)=N\I_\xi u(x),
$$
$\xi\in\mathbb R^N$ being a unit vector such that $\left\langle \hat x,\xi\right\rangle=\frac{1}{\sqrt{N}}$. Thus we have
\begin{equation}\label{eq3july6}
\I^-_Nu(x)=\frac{NC_s}{(1+|x|)^{2sq}}\int_0^{+\infty}\frac{\left(\frac{1+|x+\tau\xi|}{1+|x|}\right)^{-2sq}+\left(\frac{1+|x-\tau\xi|}{1+|x|}\right)^{-2sq}-2}{\tau^{1+2s}}\,d\tau\,.
\end{equation}
By the triangular inequality we have 
$$
\frac{1+|x\pm\tau\xi|}{1+|x|}\geq\left|\hat x\pm\frac{\tau}{1+|x|}\xi\right|\qquad\forall\tau\geq0.
$$
Then, by \eqref{eq3july6}, we infer that 
\begin{equation*}
\begin{split}
\I^-_Nu(x)&\leq\frac{NC_s}{(1+|x|)^{2sq}}\int_0^{+\infty}\frac{\left|\hat x+\frac{\tau}{1+|x|}\xi\right|^{-2sq}+\left|\hat x-\frac{\tau}{1+|x|}\xi\right|^{-2sq}-2}{\tau^{1+2s}}\,d\tau\\
&=\frac{NC_s}{(1+|x|)^{2s(q+1)}}\int_0^{+\infty}\frac{\left|\hat x+{\tau}\xi\right|^{-2sq}+\left|\hat x-{\tau}\xi\right|^{-2sq}-2}{\tau^{1+2s}}\,d\tau\\
&=\frac{NC_s}{(1+|x|)^{2s(q+1)}}c(2sq)
\end{split}
\end{equation*}
where $c(\cdot)$ is the function defined by \eqref{eqc}. Using Lemma \ref{preliminary lemma} and the assumption $2sq<\tilde\gamma$, we see that  $c(2sq)<0$. Let $v(x)=\varepsilon u(x)$  for $\varepsilon\in(0,(NC_s|c(2sq)|)^{1/(p-1)})$. Using  $q\geq\frac{1}{p-1}$ we finally obtain
\begin{equation*}
\begin{split}
\I^-_Nv(x)+v^p(x)&\leq\frac{\varepsilon}{(1+|x|)^{2s(q+1)}}\left(NC_sc(2sq)+\frac{\varepsilon^{p-1}}{(1+|x|)^{2s(qp-q-1)}}\right)\\
&\leq\frac{\varepsilon}{(1+|x|)^{2s(q+1)}}\left(NC_sc(2sq)+\varepsilon^{p-1}\right)\leq0\,,
\end{split}
\end{equation*}
completing the proof.
\end{proof}

\begin{lema}\label{lemweak}
Let $u$ be a nonnegative viscosity supersolution of
\begin{equation*}
\I_N^-u(x)=0\quad\text{in $\R^N$}.
\end{equation*}
Then the following statements hold:
\begin{itemize}
	\item there exists a positive constant $a=a(\tilde\gamma)$ such that 
\begin{equation}\label{weakdecay}
m(r)\geq a\,m(1)\,r^{-\tilde\gamma}\quad\;\forall r\geq1;
\end{equation}
\item for any $\gamma\ge\tilde\gamma$ there exists a positive constant $b=b(\gamma)$ such that 
\begin{equation}\label{weakdecay2}
m(R)\geq b\,m\left(\frac{R}{2}\right)\quad\;\forall R>0.
\end{equation}
\end{itemize} 
\end{lema}
\begin{proof}
Let $\tilde g(|x|)$ be the radial function
\begin{equation}\label{2307eq1}
\tilde g(|x|)=
\begin{cases}
\tilde f(|x|) & \text{if $|x|\leq\frac12$}\\
|x|^{-\frac{\tilde\gamma}{2}} & \text{if $|x|>\frac{1}{2}$},
\end{cases}
\end{equation}
where  $\tilde f$ is defined, for $r\geq0$,  by the formula
\begin{equation*}
f(r)=2^{\frac{\tilde\gamma}{2}}\left[-\frac16\tilde\gamma(\tilde\gamma+2)(\tilde\gamma+4)\left(r-\frac12\right)^3+\frac12\tilde\gamma(\tilde\gamma+2)r^2-\frac12\tilde\gamma(\tilde\gamma+4)r+1+\frac18\tilde\gamma(\tilde\gamma+6)\right].
\end{equation*}
With choice of such $\tilde f$, the function $\tilde g''$ is convex in $[0,+\infty)$, since the graph of $\tilde f''(r)$ is in fact the tangent line of the function $\left(r^{-\frac{\tilde\gamma}{2}}\right)^{''}=\frac{\tilde\gamma}{2}(\frac{\tilde\gamma}{2}+1)r^{-\frac{\tilde\gamma}{2}-2}$ at $r=\frac12$.

Set $w(x)=\tilde g(|x|^2)$.    By  Theorem \ref{propPk}-(ii),  for any $x\in\mathbb R^N$ we have  $$\I_N^-w(x)=N\I_\xi w(x),$$ where    $\xi\in\mathbb R^N$ is an unit vector such that $\left\langle \hat{x},\xi\right\rangle=\frac{1}{\sqrt{N}}$.
 Hence
\begin{equation}\label{2207eq1}
\I_N^-w(x)=NC_s\int_{0}^{+\infty}\frac{w(|x+\tau\xi|)+w(|x-\tau\xi|)-2w(x)}{\tau^{1+2s}}\,d\tau.
\end{equation}
If $|x|\geq1$ and $\tau>0$ it holds that 
$$
|x\pm\tau\xi|\geq\sqrt{|x|^2+\tau^2-2\frac{\tau|x|}{\sqrt{N}}}\geq|x|\sqrt{1-\frac1N}\geq\frac{1}{\sqrt{2}}.
$$
Then, using \eqref{2307eq1}-\eqref{2207eq1} and the definition of $\tilde\gamma$ given in Lemma \ref{preliminary lemma}, we infer that
$$
\I_N^-w(x)=NC_s\int_{0}^{+\infty}\frac{|x+\tau\xi|^{-\tilde\gamma}+|x-\tau\xi|^{-\tilde\gamma}-2|x|^{-\tilde\gamma}}{\tau^{1+2s}}\,d\tau=0.
$$
In this way the function
$$
\phi(x)=m(1)\frac{w(|x|)-w(R)}{w(0)-w(R)}
$$
is for any $R>1$  a solution of $\I^-_N(w,x)=0$ for $|x|\in[1,R]$.  Moreover
 $$u(x)\geq m(1)\geq\phi(x)\qquad\forall|x|\leq1$$ and $$u(x)\geq0\geq\phi(x)\qquad\forall|x|\geq R.$$ Then by comparison principle we infer that $u(x)\geq\phi(x)$ for any $|x|\in\left[1,R\right]$. Letting $R\to+\infty$  we obtain 
$$
u(x)\geq m(1)\frac{w(|x|)}{w(0)} \quad\forall|x|\geq1,
$$
which easily imply \eqref{weakdecay} with $a=\left(2^{\frac{\tilde\gamma}{2}}\left[\frac{1}{48}\tilde\gamma(\tilde\gamma+2)(\tilde\gamma+4)+1+\frac18\tilde\gamma(\tilde\gamma+6)\right]\right)^{-1}$.

The proof of \eqref{weakdecay2} follows the same idea used before. Fix $\gamma\geq\tilde\gamma$. For $R>0$, consider the function 
\begin{equation*}
\tilde g(|x|)=\begin{cases}
\tilde f(|x|) & \text{if $|x|\leq\left(\frac{R}{2\sqrt{2}}\right)^2$}\\
|x|^{-\frac{\gamma}{2}} & \text{if $|x|>\left(\frac{R}{2\sqrt{2}}\right)^2$},
\end{cases}
\end{equation*}
where 
\begin{equation*}
\begin{split}
\tilde f(r)=\left(\frac{R}{2\sqrt{2}}\right)^{-\gamma}&\left[-\frac{32}{3}\frac{\gamma(\gamma+2)(\gamma+4)}{R^6}\left(r-\frac{R^2}{8}\right)^3+\frac{8}{R^4}\gamma(\gamma+2)r^2\right.\\&\quad-\frac{2}{R^2}\gamma(\gamma+4)r+1+\frac{\gamma}{8}(\gamma+6)\Bigg].
\end{split}
\end{equation*}
Set $w(x)=\tilde g (|x|^2)$. Since $\tilde g''$ is convex, we are in position to use the representation formula \eqref{2207eq1}. Taking into account that for $|x|\geq\frac R2$  
$$
|x\pm\tau\xi|\geq\sqrt{|x|^2+\tau^2-2\frac{\tau|x|}{\sqrt{N}}}\geq|x|\sqrt{1-\frac1N}\geq\frac{R}{2\sqrt{2}},
$$
then
$$
\I_N^-w(x)=NC_s\int_{0}^{+\infty}\frac{|x+\tau\xi|^{-\gamma}+|x-\tau\xi|^{-\gamma}-|x|^{-\gamma}}{\tau^{1+2s}}\,d\tau\geq0
$$
the last inequality being a consequence of the fact that  $\gamma\geq\tilde\gamma$. Consider now the function
$$
\phi(x)=m\left(\frac R2\right)\frac{w(|x|)-w(2R)}{w(0)-w(2R)},
$$
which is in turn a solution of $\I^-_N\phi(x)\geq0$ for $|x|\in[\frac{R}{2},2R]$ and satisfies
$$u(x)\geq m(1)\geq\phi(x)\qquad\forall|x|\leq\frac R2$$ and $$u(x)\geq0\geq\phi(x)\qquad\forall|x|\geq 2R.$$ By comparison principle we conclude 
\begin{equation*}
m(R)\geq m\left(\frac R2\right)\frac{w(R)-w(2R)}{w(0)-w(2R)}=m\left(\frac R2\right)\frac{1-2^{-\gamma}}{(2\sqrt{2})^{\gamma}\left[1+\frac{\gamma}{48}\left((\gamma+2)(\gamma+4)+6(\gamma+6)\right)\right]}\,.
\end{equation*}
\end{proof}

\medskip

\begin{proof}[\textbf{Proof of Theorem \ref{Liouville-N}}]
We shall detail the proof in the critical case $p=1+\frac{2s}{\tilde\gamma}$, since if $p>1+\frac{2s}{\tilde\gamma}$  the conclusion follows by Proposition \ref{propexi}, while the subcritical case $p<1+\frac{2s}{\tilde\gamma}$ can be treat in the same way as we did in the proof of Theorem \ref{teoLiouville+}, using now Lemma \ref{lemweak}. When $p=1+\frac{2s}{\tilde\gamma}$ we need some extra work. In particular we are not in position to use the analogous of Lemma \ref{2904lem1} for the operator $\I_N^-$, due to the lack of validity of the representation formula for $\Gamma(|x|)=\frac{\ln|x|}{|x|^{\tilde\gamma}}$. Note that $\Gamma$ doesn't even belong to $L^1_{loc}(\mathbb R^N)$ when $N\geq3$, since $\tilde\gamma>1$ (see Remark \ref{rmktildegamma}). Moreover moreover $\tilde\Gamma$ is  concave  near the origin. On the other hand, for $x$ far away the origin, we shall still obtain some useful informations that are sufficient to conclude.

Let $\tilde\Gamma(|x|)=\frac12\frac{\ln|x|}{|x|^{\tilde\gamma/2}}$. The function $\tilde\Gamma''(r)$ is convex for $r\geq r_0:=\exp\left(\frac{2}{\tilde\gamma}+\frac{2}{\tilde\gamma+2}+\frac{2}{\tilde\gamma+4}+\frac{2}{\tilde\gamma+6}\right)$. Let $\tilde f''(r)=\tilde\Gamma''(r_0)+\tilde\Gamma'''(r_0)(r-r_0)$ be the tangent line of $\tilde\Gamma''$ at $r=r_0$. By construction the function 
$$
\tilde g(|x|)=\begin{cases}
\tilde f(|x|) & \text{if $|x|\leq r_0$}\\
\tilde\Gamma(|x|)   & \text{if $|x|>r_0$},
\end{cases}
$$
where 
\begin{equation*}
\begin{split}
\tilde f(|x|)&=\frac16\tilde\Gamma'''(r_0)(r-r_0)^3+\frac12\tilde\Gamma''(r_0)r^2 +\left(\tilde\Gamma'(r_0)-\tilde\Gamma''(r_0)r_0\right)r\\
&\quad+\tilde\Gamma(r_0)-\frac12\tilde\Gamma''(r_0)r_0^2-\left(\tilde\Gamma'(r_0)-\tilde\Gamma''(r_0)r_0\right)r_0,
\end{split}
\end{equation*}
is such that $\tilde g''$ is convex in $[0,+\infty)$.  Hence, setting $w(x)=\tilde g(|x|^2)$
 and using Theorem \ref{propPk} we have 
$$
\I_N^-w(x)=N\I_{\xi}w(x),
$$
$\xi\in\mathbb R^N$ being an unit vector such that $\left\langle \hat x,\xi\right\rangle=\frac{1}{\sqrt{N}}$.
Moreover for $|x|\geq\sqrt{2r_0}$ it holds that $|x\pm\tau\xi|\geq\sqrt{r_0}$ for any $\tau>0$. Then for any $|x|\geq\sqrt{2r_0}$
\begin{equation*}
\begin{split}
\I_N^-w(x)&=NC_s\int_0^{+\infty}\frac{\Gamma(|x+\tau\xi|)+\Gamma(|x-\tau\xi|)-2\Gamma(|x|)} {\tau^{1+2s}}\,d\tau\\
&=NC_s\left(\ln|x|\int_0^{+\infty}\frac{{|x+\tau\xi|^{-\tilde\gamma}}+{|x-\tau\xi|^{-\tilde\gamma}}-{|x|^{-\tilde\gamma}}} {\tau^{1+2s}}\,d\tau\right.\\
&\quad\quad\qquad\left.\frac{1}{|x|^{\tilde\gamma+2s}}\int_0^{+\infty}\frac{\frac{\ln|\hat x+\tau\xi|}{|\hat x+\tau\xi|^{\tilde\gamma}}+\frac{\ln|\hat x-\tau\xi|}{|\hat x-\tau\xi|^{\tilde\gamma}}}{\tau^{1+2s}}\,d\tau\right)\\
&=NC_s\frac{1}{2|x|^{\tilde\gamma+2s}}\int_0^{+\infty}\left({\frac{\ln\left(1+\tau^2+\frac{2\tau}{\sqrt{N}}\right)}{\sqrt{1+\tau^2+\frac{2\tau}{\sqrt{N}}}^{\,\tilde\gamma}}+\frac{\ln\left(1+\tau^2-\frac{2\tau}{\sqrt{N}}\right)}{\sqrt{1+\tau^2+\frac{2\tau}{\sqrt{N}}}^{\,\tilde\gamma}}}\right){\tau^{-(1+2s)}}\,d\tau\\
&\geq -\frac{C}{|x|^{\tilde\gamma+2s}}
\end{split}
\end{equation*}
where $C=C(N,s)$ is a positive constant.\\
 Now for $r_2>r_1>\sqrt{2r_0}$ we consider the function
$$
\phi(x)=m(r_1)\frac{w(|x|)-w(r_2)}{w(0)-w(r_2)}\qquad\forall|x|\in[r_1,r_2].
$$
Without loss of generality we may further assume that $w(0)-w(r_2)>\frac12w(0)$, so that 
\begin{equation}\label{2407eq3}
\I_N^-\phi(x)\geq-Cm(r_1)\frac{1}{|x|^{\tilde\gamma+2s}}
\end{equation}
 where $C$ is a positive constant depending only on $N$ and $s$. In addition $u(x)\geq\phi(x)$ for any $|x|\in[0,r_1]\cup[r_2,+\infty)$.\\
Using the equation \eqref{2407eq1} and \eqref{weakdecay} we also have 
\begin{equation}\label{2407eq4}
\I_N^-u(x)\leq-(m(|x|))^{1+\frac{2s}{\tilde\gamma}}-\leq\left(am(1)\right)^{1+\frac{2s}{\tilde\gamma}}\frac{1}{|x|^{\tilde\gamma+2s}}\qquad\forall|x|\geq1.
\end{equation}
Since $m(r_1)\to0$ as $r_1\to+\infty$, in view of the inequality 
\begin{equation}\label{2407eq2}
m(R)R^{\tilde\gamma}<C\qquad\forall R>0,
\end{equation}
for some positive constant $C$, by \eqref{2407eq3}-\eqref{2407eq4}
 we can then pick $r_1$ sufficiently large such that 
$$
\I_N^-u(x)\leq\I_N^-\phi(x)\qquad\forall|x|\in[r_1,r_2].
$$
By comparison principle we have $u(x)\geq\phi(x)$ for any $|x|\in[r_1,r_2]$. Letting $r_2\to+\infty$ we deduce that 
$$
m(r)\geq\frac{m(r_1)}{w(0)}w(r)=\frac{m(r_1)}{w(0)}\frac{\ln r}{r^{\tilde\gamma}}\qquad\forall r>r_1,
$$ 
leading to a contradiction to \eqref{2407eq2} in the limit  as $r\to+\infty$.
\end{proof}



\section{On the operator $\JJ_k^{\pm}$}
\label{secJ}

In this section we concentrate on the operators $\JJ_k^{\pm}$ defined in~\eqref{operador2}. We leave off the analysis the cases $k=1$, where $\JJ_{1}^{\pm}$ meets $\I_1^{\pm}$ studied in the previous sections, and $k = N$, case in which $\JJ_k^{\pm} = -(-\Delta_{\R^N})^s$, already studied in~\cite{FQ}. For simplicity, we write $\Delta_{\R^k}^s = -(-\Delta_{\R^k})^s$, to denote the fractional Laplacian in $\R^k$.

\medskip

The key technical result of this section is the following
\begin{prop}\label{propJ}
	Assume $1 < k < N$. Let $u(x) = \tilde g(|x|^2)$ be such that $u \in L^1_{k,2s}$. 
\begin{itemize}
	\item[(i)] If $\tilde g$ is convex, then
	$$
	\JJ_k^- u(x)=\JJ_{V}u(x),
	$$ 
	where $V$ is any $k$ dimensional subspace which is orthogonal $x$.
	
	\item[(ii)] If $\tilde g\in C^1([0,+\infty))$ and $\tilde g$ is convex, then
	$$
	\JJ_k^+u(x)= \JJ_{V}u(x),
	$$
	where $V$ is any $k$-dimensional subspace containing $x$.
\end{itemize}

Moreover, defining $\tilde u: \R^k \to \R$ as $\tilde u(y):=g(|y|)$, 
we have
\begin{equation}\label{formula}
\JJ_k^+ u(x) = \Delta_{\R^k}^s \tilde u(y),
\end{equation}
where $y \in \R^k$ is such that $|y| = |x|$.
\end{prop}
\begin{proof}
We start with (i). Let  $V = \langle \{ \xi_1, \dots, \xi_k \} \rangle$. 
Recalling that $K(\tau) = (\sum_{i=1}^{k} \tau_i^2)^{-\frac{N + 2s}{2}}$ and dropping the P.V. we can write
\begin{equation*}\label{J}
\begin{split}
\JJ_V u(x) 
&=  \frac{C_{k,s}}{2} \int_{\R^k} [\tilde g(|x + \sum_{i=1}^k \tau_i\xi_i|^2) + g(|x - \sum_{i=1}^k \tau_i\xi_i|^2)- 2\tilde g(|x|^2)] K(\tau)d\tau \\
& = \frac{C_{k,s}}{2} \int_{\R^k} [\tilde g(|x|^2 + |\tau|^2 + 2\sum_{i=1}^k \tau_i \left\langle \xi_i, x\right\rangle + g(|x|^2 + |\tau|^2 - 2\sum_{i=1}^k \tau_i \left\langle \xi_i, x\right\rangle)- 2\tilde g(|x|^2)] K(\tau)d\tau\\
&\geq  C_{k,s} \int_{\R^k} [\tilde g(|x|^2 + |\tau|^2) - \tilde g(|x|^2)] K(\tau)d\tau \\
&= \JJ_W u(x),
\end{split}
\end{equation*}
where in the inequality above we have used the convexity inequality ~\eqref{hola}, with $a=|x|^2 + |\tau|^2$ and $b=2\sum_{i=1}^k \tau_i \left\langle \xi_i, x\right\rangle $, and $W$ is any $k$-dimensional subspace orthogonal to $x$.

\smallskip

Now we deal with (ii). Assume that $x\not\in V$ and let $\xi_{k+1}$ be the unit vector that is orthogonal to $V$ in the $k+1$ dimensional space generated by $\xi_1,\dots,\xi_k$ and $x$. 

We define $\hat x_{V^\perp} := \left\langle \hat x,\xi_{k+1}\right\rangle\xi_{k+1}=\cos(\theta) \xi_{k+1}$ with a choice of $\xi_{k + 1}$ such that $\theta \in [0, \pi/2)$. Thus,
it is possible to write $\hat x = \hat x_V + \hat x_{V^\perp}$ with $\hat x_V \in V$ such that $|\hat x_V|=\sin(\theta)$. 


Observe that the integral which defines $\JJ_V$ is independent of the choice of the vectors $\xi_i$ generating $V$, so without loss of generality we can assume that  $\{ \xi_1, ..., \xi_k \}$ are such that $\xi_i \perp \hat x$ for all $i=2,...,k$.

With this choice we get the expression
\begin{align*}
\JJ_V u(x) = & \frac{C_{k,s}}{2} \int_{\R^k} \Big{[}   \tilde g(|x|^2 + \tau^2 + 2|x|\tau_1\sin(\theta)) 
+ \ \tilde g(|x|^2 + \tau^2 - 2 |x|\tau_1 \sin(\theta)) 
- 2 \tilde g(|x|^2) \Big{]} K(\tau)d\tau\\
=: & \frac{C_{k,s}}{2} f(\theta)
\end{align*}

Notice that $f \in C([0, \pi/2]) \cap C^1([0, \pi/2))$. Then, for each $\phi \in [0, \theta)$ we have
\begin{align*}
f'(\theta)= 2\cos(\theta) |x| \int_{{ \R^k}} \left[\tilde g'(|x|^2 + \tau^2 + 2|x|\tau_1\sin(\theta))
-  \tilde g'(|x|^2 + \tau^2 - 2 |x|\tau_1 \sin(\theta))\right]\tau_1
K(\tau)d\tau\,.
\end{align*}



Since $\tilde g'$ is nondecreasing, and the function inside the integral is even, we conclude that $f$ is nondecreasing. Thus
$$
\JJ_k^+u(x) = \frac{C_{k,s}}{2} f(\pi/2) = \JJ_{V} u(x)
$$ 
for any $k$-th dimensional space $V$ containing $x$.

In order to prove \eqref{formula}, let $V = \left\langle \{ \xi_i \}_{i=1}^k\right\rangle$ a $k$-dimensional subspace containing $x$. In this way  $\JJ_k^+u(x)=\JJ_Vu(x)$ and without loss of generality we can further assume that $\xi_i \perp \hat x$ for all $i=2,...,k$. 
Let $\{ e_i \}_{i=1}^{k}$ the canonical basis in $\R^k$. Using the rotation invariance of the fractional Laplacian, for  $y \in \R^k$ such that $|y|=|x|$ we have
\begin{align*}
\Delta_{\R^k}^s \tilde u (y) & = \Delta_{\R^k}^s \tilde u (|y|e_1) \\
& = \frac{C_{k,s}}{2}\,\mathrm{P. V.} \int_{\R^k} [ \tilde g(||y| e_1 + \sum_{i=1}^{k} z_i e_i|^2) + \tilde g(||y| e_1 - \sum_{i=1}^{k} z_i e_i|^2) - 2 \tilde g(|y|^2)]  |z|^{-(k + 2s)}dz \\
& =  \frac{C_{k,s}}{2}\,\mathrm{P. V.} \int_{\R^k} [\tilde g(|x|^2 + 2 |x| \tau_1 + |\tau|^2) + \tilde g(|x|^2 - 2|x| \tau_1 + |\tau|^2) - 2 \tilde g(|x|^2)] K(\tau)d\tau\\
& = \JJ_V u (x)  = \JJ_k^+ u (x),
\end{align*}
as we wanted to show.
\end{proof}

Using known results for the fractional Laplacian (see~\cite{BG, CSext}) and the previous proposition we get the following
\begin{cor}
The function $u(x)=|x|^{-(k-2s)}$ satisfies 
$$
\JJ_k^+ u(x)=0 \quad \mbox{for} \ x \in \R^N \setminus \{ 0 \}.
$$
\end{cor} 

Using the representation formula and Theorem 1.3 in Felmer and Quaas~\cite{FQ}, we can get the Liouville Theorem for $\JJ_k^+$
\begin{teo}\label{teoLiouvilleJ}
	Let $1 < k < N$. Then, the equation 
	\begin{equation}\label{eqfinal}
	\JJ_k^+ u(x)+u^p(x)=0\quad\text{in $\R^N$}
	\end{equation}
	has nontrivial viscosity supersolutions if, and only if,
	$$
	p>\frac{k}{k - 2s}.
	$$
\end{teo}

\begin{proof}
	For the existence of nontrivial supersolution, we consider $\frac{1}{p - 1} < q < \frac{k - 2s}{2s}$ and $v(y) = (1 + |y|)^{-2sq}, y \in \R^k$. According to~\cite{FQ}, we have
	$$
	\JJ_k^+ v(x) + v^p = \Delta^s_{\R^k} v(y) + v^p \leq 0.
	$$
	
	On the other hand, if $p > k/(k - 2s)$  and there exists a nontrivial supersolution $u$ for~\eqref{eqfinal}, the function 
	$$
	v(x) = \min \{ u(Ox) : O \ \mbox{is a rotation matrix in} \ \R^N \}
	$$
	is a positive, radial supersolution for~\eqref{eqfinal}. Let $\tilde h: [0,+\infty) \to \R$ such that $v(x) = \tilde h(|x|)$ and denote $w(y) = \tilde h(|y|), y \in \R^k$. Then,  we have
	\begin{align*}
	\Delta_{\R^k}^s w(y) & = \frac{C_{k,s}}{2} \int_{\R^k} [\tilde h(|\sum_{i=1}^{k} y_i e_i + \sum_{i=1}^{k} z_i e_i|) + \tilde h(|\sum_{i=1}^{k} y_i e_i - \sum_{i=1}^{k} z_i e_i|) - 2\tilde h(|y|)]K(z)dz \\
	& \leq \JJ_k^+ v (\sum_{i=1}^{k} y_i e_i),
	\end{align*}
	where $\{ e_i \}_{i=1}^{N}$ is the canonical basis (we have identified $e_i \in \R^k$ for $i \leq k$). Then, $w$ is a nontrivial supersolution to
	$$
	\Delta_{\R^k}^s w + w^p \leq 0 \quad \mbox{in} \ \R^k,
	$$
	which contradicts the nonexistence result in~\cite{FQ}.
\end{proof}

For $\JJ_k^-$, in analogy to Theorem~\ref{LiouvilleMinus} we have the following
\begin{teo}\label{LiouvilleMinusJ}
	Assume $1 < k < N$ and consider the equation
	\begin{equation}\label{equazione}
	\JJ_k^- u(x)+u^p(x)=0\quad\text{in $\R^N$.}
	\end{equation}
	Then
	\begin{itemize}
		\item[(i)] for any $p\geq1$ there exist positive classical solutions of \eqref{equazione};
		\item[(ii)] for any $p\in(0,1)$ there exist nonnegative viscosity solutions $u\not\equiv0$ of \eqref{equazione}.
	\end{itemize}

\end{teo}

\section{Appendix}

In this section we provide a sketch of the proof of some results related to the convergence of the nonlocal operators presented here towards the local regime, that is, when $s \to 1^-$. It is worth to mention that the normalizing constant in~\eqref{frac} is given (see~\cite{Hitch}) by
\begin{equation}\label{normconst}
C_{N,s} = \Big{(} \int_{{ \R^N}} \frac{1 - \cos(z_1)}{|z|^{N + 2s}}dz \Big{)}^{-1}.
\end{equation}

We start with the following convergence result that is at the core of the stability of viscosity solutions. Recall that we denote $C_s=C_{1,s}$.
\begin{lema}
	Let $r_0\in(0,1)$ and let $u \in C^2(\R^N) \cap L^1_{1,2s}$ for any $s\in(r_0,1)$ be such that 
	$$
	\| u \|_{L^1_{1,2s}}:=\sup_{\dim(V)=1}\left\{  \int_{V} \frac{|u(y)| d  \mathcal H^1(y)}{1 + |y|^{1 + 2s}}  \right\} \leq C
	$$
	for $C$ independent of $s$. Then, for each $x \in \R^N$ we have
	$$
	\I_k^\pm u(x) \to \PP_k^{\pm}u(x) \quad \mbox{as} \ s \to 1^-.
	$$

	Analogously, if  $u \in C^2(\R^N) \cap L^1_{k,2s}$ for any $s\in(r_0,1)$ and 
	$$
	\| u \|_{L^1_{k,2s}}:=\sup_{\dim(V)=k}\left\{  \int_{V} \frac{|u(y)| d  \mathcal H^k(y)}{1 + |y|^{k + 2s}}  \right\} \leq C
	$$
	for $C$ independent of $s$, then for each $x \in \R^N$ we have $\JJ_k^{\pm}u(x) \to \PP_k^{\pm} u(x)$ as $s \to 1^-$.
\end{lema}

\begin{proof}
	We write the result for $\I_k^+$, being the result for $\I_k^-$ analogous. 
	
	For $s\in(r_0,1)$, there exists a frame $\{ \xi_j^s \}_j$ such that
	\begin{align*}
	\I_k^+ u(x) - \PP_k^+u(x) = \sum \limits_{j=1}^k \I_{\xi_j^s} u(x) - \PP_k^+u(x)
	 \leq \sum \limits_{j=1}^k \Big{(} \I_{\xi_j^s} u(x) - \langle D^2u(x) \xi_j^s, \xi_j^s \rangle \Big{)}.
	\end{align*}
	
	Let $\varepsilon > 0$. For $\delta \in (0,1)$ to be fixed, we can write for each $j$
	\begin{align*}
	\I_{\xi_j^s} u(x) = \frac{1}{2} C_s \int_{-\delta}^\delta  \langle D^2u(\tilde x_{\tau, j, s}) \xi_j^s, \xi_j^s \rangle  |\tau|^{1 - 2s} d\tau + C_s O(\delta^{-2s}),
	\end{align*}
	where $\tilde x_{\tau, j, s} \in B_\delta(x)$ for all $\tau, j, s$, and $O(\delta^{-2s})$ just depend on the $\| u \|_{L^1_{2s}} \leq C$ for some $C$ independent of $s$. Using the continuity of $D^2 u$, we can fix $\delta$ small enough in order to have
	$$
	|D^2u(\tilde x_{\tau, j, s}) - D^2 u(x) | \leq \varepsilon, 
	$$
	for all $s \in (r_0,1), j=1,\ldots,k$ and $|\tau| < \delta$.
Then, we can write
	\begin{align}\label{0804eq3}
	\I_k^+ u(x) - \PP_k^+u(x) 
	\leq & \frac{C_s k \varepsilon}{2 - 2s} \delta^{2 - 2s} + \sum \limits_{j=1}^k \langle D^2 u(x) \xi_j^s, \xi_j^s \rangle \Big{(} \frac{C_s}{2 - 2s} \delta^{2 - 2s} - 1 \Big{)} + C_s O(\delta^{-2s}). 
	\end{align}
	
	Since
	$$
	\frac{C_s}{2(1 - s)} \to 1 \quad \mbox{as} \ s \to 1^-,
	$$
	(see \cite{Hitch}),
	then passing to the limit in \eqref{0804eq3}, we have 
	$$
	\limsup_{s\to1^-}\I_k^+ u(x) - \PP_k^+u(x)\leq k\varepsilon. 
	$$
	%
	A reverse inequality can be found in the same way, and the result follows.
	
	For $\JJ_{k}^{\pm}$ the proof is similar, so we will be sketchy. In this case, given $V = \{ \xi_i \}_{i=1}^k \in \mathcal V_k$ and for $\delta > 0$ we can write
	\begin{align*}
	\JJ_V u(x) = \frac{C_{k,s}}{2} \sum_{i,j=1}^k \int_{B_\delta} \langle D^2u(\tilde x_{i,j,s,\tau})  \xi_i, \xi_j \rangle \tau_i \tau_j |\tau|^{-(k + 2s)} d\tau + C_{k,s}\, O(\delta^{-2s}),
	\end{align*}
	where $\tilde x_{i,j,s,\tau} \in \left\langle V\right\rangle$  is such that $|\tilde x_{i,j,s,\tau} - x| \leq \delta$. Then, using the continuity of $u$, for each $\varepsilon > 0$ we can get $\delta > 0$ such that
	\begin{align*}
	\JJ_V u(x) = & O(\varepsilon) + \frac{C_{k,s}}{2} \sum_{i,j=1}^k \langle D^2u(x)  \xi_i, \xi_j \rangle \int_{B_\delta}  \tau_i \tau_j |\tau|^{-(k + 2s)} d\tau + C_{k,s}\, O(\delta^{-2s}),
	\end{align*}
	and using the symmetry of the integral term, we have
	\begin{align*}
	\int_{B_\delta}  \tau_i \tau_j |\tau|^{-(k + 2s)} d\tau = \delta_{ij} \int_{B_\delta}  \tau_1^2 |\tau|^{-(k + 2s)} d\tau = \delta_{ij} k^{-1} \int_{B_\delta} |\tau|^{2 - k - 2s} d\tau = \delta_{ij} k^{-1} |\mathbb S^{k - 1}| \frac{\delta^{2 - 2s}}{2 - 2s}.
	\end{align*}
	where $\delta_{ij}$ is the Kronecker delta, and $|\mathbb S^{k - 1}|$ denotes the $(k-1)$-dimensional measure of the unit sphere in $\R^k$. For $k > 1$, we have the estimate  (see Corollary 4.2 in~\cite{Hitch})
	\begin{equation*}
	\frac{C_{k,s}\, |\mathbb S^{k-1}|}{4 k (1 - s)} \to 1 \quad \mbox{as} \ s \to 1^-,
	\end{equation*}
	we conclude the result. 
\end{proof}

\begin{lema}\label{lemaasym}
Let $\bar \gamma=\bar\gamma(k,s)$ defined in Proposition~\ref{propbargamma}. Then, $\bar \gamma \to 0$ as $s \to 1^-$.
		
		Let $\tilde \gamma=\tilde \gamma(N,s)$ defined in Lemma~\ref{preliminary lemma}. Then, $\tilde \gamma \to N - 2$ as $s \to 1^-$.
\end{lema}
\begin{proof}
	We already know that $\bar \gamma\in(0,1)$ and  $\tilde \gamma > 0$. Moreover from the proof of Lemma \ref{preliminary lemma} we can also infer that, for any $s\in(\frac12,1)$, $\tilde\gamma<c$ where $c$ is a positive constant depending only on $N$. Hence both $\bar\gamma$ and $\tilde\gamma$ are uniformly bounded.
	
	For $\bar \gamma$, let us first observe that, by Proposition \ref{propbargamma}, one has $$\bar\gamma(k,s)<\bar\gamma(k+1,s).$$ Then it is sufficient to prove that $\bar \gamma\to0$, as $s\to1^-$, for $k$ large, say $k\geq4$.  If not, let $\gamma_1 \in(0,1]$ be an accumulation point of $\bar \gamma$ as $s \to 1^-$. Then, by stability of viscosity solutions, the function $w_{\gamma_1}(x)=|x|^{-\gamma_1}$ would be a solution of $\PP_k^+(D^2w)=0$ for $x\neq0$. But this contradicts  the fact that the only positive exponent $\gamma$ such that $w_\gamma(x)=|x|^{-\gamma}$ is  solution for $\PP_k^+$  is $\gamma=k-2$, see \cite{BGL}, while $\gamma_1<k-2$ for $k\geq4$. Thus, $\bar \gamma \to 0$ as $s \to 1^-$.
	
	\medskip
	
	On the other hand, let $\gamma_1 \geq 0$ be an accumulation point of $\tilde \gamma$ as $s \to 1^-$. Using the definition of $c(\gamma)$ in~\eqref{eqc}, for each $s$ we have
	\begin{equation*}
	0 = C_s \int_0^{+\infty}\frac{\left(1+\tau^2+\frac{2}{\sqrt{N}}\tau\right)^{-\tilde \gamma/2}+\left(1+\tau^2-\frac{2}{\sqrt{N}}\tau\right)^{-\tilde \gamma/2}-2}{\tau^{1+2s}}\,d\tau,
	\end{equation*}
	and from here we have
	\begin{align*}
	0 = C_s \int_0^{\frac{1}{2\sqrt{N}}}\frac{\left(1+\tau^2+\frac{2}{\sqrt{N}}\tau\right)^{-\tilde \gamma/2}+\left(1+\tau^2-\frac{2}{\sqrt{N}}\tau\right)^{-\tilde \gamma/2}-2}{\tau^{1+2s}}\,d\tau + C_s O(1),
	\end{align*}
	where $O(1)$ is independent of $s$ close to $1$. By a Taylor expansion, we have
	\begin{equation*}
	\left(1+\tau^2+\frac{2}{\sqrt{N}}\tau\right)^{-\tilde \gamma/2}+\left(1+\tau^2-\frac{2}{\sqrt{N}}\tau\right)^{-\tilde \gamma/2}-2 = \tilde \gamma \tau^2 \Big{(}-1 + \frac{\tilde \gamma + 2}{N}\Big{)} + O(\tau^3),
	\end{equation*}
	where $O(\tau^3)$ is independent of $s$. Thus, replacing this into the integral term we get
	\begin{align*}
	0 = \frac{\tilde\gamma }{(4N)^{1-s}} \Big{(}-1 + \frac{\tilde \gamma + 2}{N}\Big{)} \frac{C_s}{2 - 2s} + C_s O(\frac{1}{3 - 2s}) + C_s O(1),
	\end{align*}
	from which, taking limit as $s \to 1^-$ we arrive at 
	$$
	0 =  \gamma_1 \Big{(} -1 + \frac{\gamma_1 + 2}{N} \Big{)},
	$$
	for some $C > 0$, from which the result follows. If $N \geq 3$ we know that $\gamma_1 \geq 1$  (see Remark~\ref{rmktildegamma}), from which the result follows. In the case $N=2$, we see that $\gamma_1 = 0$.
\end{proof}

%


\noindent
{\bf Acknowledgements:} 
E. T. was partially supported by Fondecyt no. 1201897 and I. B. and G. G. are partially supported by GNAMPA-INdAM. This work was started when E. T. was visiting Sapienza as a 3 months visiting professor for 2020, the three authors wish to thank Sapienza University for the support.

\end{document}